\documentclass{article}
\usepackage{amsmath}
\usepackage[normalem]{ulem}
\usepackage[left=2.5cm,top=3cm,right=2.5cm,bottom=3cm,bindingoffset=0.5cm]{geometry}
\usepackage[dvipsnames]{xcolor}

\usepackage{fontenc}
\usepackage[utf8]{inputenc}
\usepackage{amsthm}

\usepackage{bbm}
\usepackage{amssymb}
\usepackage{physics}
\usepackage[
backend=biber,
style=alphabetic,
sorting=nyt,
maxbibnames=99,
giveninits=true,
]{biblatex}
\usepackage{enumitem}
\usepackage{hyperref}
\usepackage{xcolor}
\usepackage{soul}
\usepackage{cleveref}

\numberwithin{equation}{section}

\newcommand{\R}{\mathbbm{R}}
\newcommand{\E}{{\mathcal{E}}}
\newcommand{\V}{\mathbbm{V}}
\newcommand{\Y}{\mathbbm{Y}}

\newcommand{\K}{\tilde{\mathcal{K}}(\Phi)}
\newcommand{\KK}{\mathcal{K}(\Phi)}

\newcommand{\xiz}{\Big(\xi+\frac{|\zeta|^2}{2}\Big)}

\newcommand{\C}{\mathcal{C}}
\newcommand{\de}{\,\mathrm{d}}
\newcommand{\f}[1]{\mathbf{#1}}
\DeclareMathOperator*{\argmin}{arg\,min}
\DeclareMathOperator*{\diiv}{div}
\renewcommand{\t}{\partial_t}
\DeclareMathOperator{\dom}{dom}
\newcommand{\sym}{\mathrm{sym}}
\newcommand{\BV}{\mathrm{BV}}
\newcommand{\N}{\mathbbm N}

\newcommand{\ol}[1]{\overline{#1}}

\renewcommand{\curl}{\nabla\times}

\newtheorem{theorem}{Theorem}[section]

\newtheorem{lemma}[theorem]{Lemma}
\newtheorem{proposition}[theorem]{Proposition}
\theoremstyle{definition}
\newtheorem{definition}[theorem]{Definition}
\newtheorem{remark}[theorem]{Remark}

\title{Existence and selection of solutions in the energy-variational framework with applications in fluid dynamics
\footnotetext{%
\textit{Keywords:} 
global existence, selection criteria, energy-variational solutions, abstract evolution equations, time discretization, Euler--Korteweg system, binormal curvature flow, vortex filament equation.}
\footnotetext{%
\textit{MSC 2020: }
35A01,
35A15,
35D99,
35Q35,
76B47,
76N10
}}
\author{Thomas Eiter\thanks{Freie Universit\"at Berlin, Department of Mathematics and Computer Science, Arnimallee 14, 14195 Berlin, Germany.} \textsuperscript{,}%
\thanks{Weierstrass Institute for Applied Analysis and Stochastics, Anton-Wilhelm-Amo-Str.~39, 10117 Berlin, Germany. \\
Email: thomas.eiter@wias-berlin.de\\
\phantom{Email: }robert.lasarzik@wias-berlin.de\\
\phantom{Email: }marcel.sliwinski@wias-berlin.de} \ \quad 
Robert Lasarzik\footnotemark[2] \quad Marcel Śliwiński\footnotemark[2]}
\date{}

\begin{document}
\maketitle
\begin{abstract}
    We provide a novel existence result for energy-variational solutions to a general class of evolutionary partial differential equations. 
    Compared to previous works on this solution concept, the generalization is mainly twofold:
    a relaxation of the assumptions on the regularity weight and the admissibility of energies with merely linear growth. 
    We apply the abstract theory to the Euler--Korteweg 
    system 
    and to the equation for binormal curvature flow,
    which serve as examples that require the first and second generalization, respectively.
    Moreover, we discuss criteria 
    that are suitable for the selection of 
    particular energy-variational solutions
    in the possibly multi-valued solution set. 
\end{abstract}

\tableofcontents

\section{Introduction}

Ideas on weak solution concepts for partial differential equations (PDEs) can be traced back to Lagrange's study of the wave equation in 1760~\cite{lagrange},
and the seminal work by Leray~\cite{leray} on weak solutions to the Navier--Stokes equations
was the first such theory developed in a functional analytic framework. 
Nowadays, even further generalized notions of solutions have become common in 
mathematical research on PDEs, 
in particular in fluid mechanics.
For the Euler equations, there are different solvability frameworks 
that allow for establishing existence results.
The first global existence result for the incompressible Euler equations
is due to DiPerna and Majda~\cite{DiPernaMajda}, who used generalized Young measures 
to describe the limits of approximate solutions
in terms of \textit{measure-valued solutions}.
A different concept was introduced by Pierre-Louis Lions~\cite[Sec.~4.4]{lionsbook}
in the form of \emph{dissipative solutions}, based on a variational inequality, the relative energy inequality. 
These concepts have been further developed and applied to more involved equations. 
For instance, there are results concerning the
existence of measure-valued solutions to the
compressible isentropic Euler system~\cite{japMeasVal}, the Euler--Korteweg system and 
the Euler--Korteweg--Poisson system~\cite{gallenmüller2023cahnhillardkellersegelsystemshighfriction}, as well
as dissipative solutions to the Ericksen--Leslie equations~\cite{diss19} and a model for
magnetohydrodynamics~\cite{raymond}.
In this article, we focus on the concept of
\emph{energy-variational solutions}.
As has been shown before, it is suitable to derive general existence results for large classes of
systems of conservation laws~\cite{envarhyp}, 
including the Euler equations for incompressible and compressible fluids,
of more general damped Hamiltonian systems~\cite{viscoenvar},
and of models for visco-elasto-plasticity~\cite{eiter2025viscoelastoplasticity}.

Consider two Banach spaces $\Y$ and $\V$ with a dense and compact embedding $\mathbbm{Y}\hookrightarrow\mathbbm{V}$,
and an abstract evolutionary problem of the form
\begin{subequations}
\label{eq:general}
\begin{alignat}{2}
    \partial_t U(t) + A(t,U(t))&=0 \qquad &&\text{in } \mathbbm{Y}^* \text{ for }t\in(0,T), \label{gen eq}\\
    U(0,\cdot)&=U_0 \qquad &&\text{in } \mathbbm{V}^* ,\label{gen eq 1}
\end{alignat}
\end{subequations}
where $U:[0,T] \rightarrow \mathbbm{V}^*$  is the unknown state variable and $A$ is a (time-dependent) operator.
We assume the existence of a convex energy $\E$ and a dissipation potential $\Psi$ 
such that any smooth solution to \eqref{eq:general} satisfies the energy-dissipation balance
\begin{equation}
    \label{eq:energy.dissipation.intro}
    \mathcal{E}(U)\Big|^t_s + \int_s^t \Psi(\tau,U(\tau)) \dd{\tau} = 0
\end{equation}
for $s,t\in[0,T]$ with $s<t$.
An energy-variational solution consists of the state variable $U$ and 
a second variable $E\colon[0,T]\to[0,\infty)$ such that $\E(U)\leq E$ and
the variational inequality
\begin{equation}\label{envarintro}
[ E-\langle U, \Phi \rangle ]  \Big|^t_s +\int_s^t \langle U , \partial_t \Phi \rangle- \langle A(\tau,U), \Phi\rangle +\Psi(\tau,U) +\mathcal{K}(\Phi)[\mathcal{E}(U)-E] \dd{\tau} \leq 0
\end{equation}   
holds for all suitable test functions $\Phi \in \mathcal{C}^1([0,T];\mathbbm{Y})$ 
and almost all $s < t \in [0,T]$, including $s=0$ with $U(0)=U_0$.
Here, the \textit{regularity weight} $\mathcal K$ is chosen in a way
that ensures the weak* lower semicontinuity 
of the left-hand side of~\eqref{envarintro} 
with respect to the natural topology.
The energy-variational inequality~\eqref{envarintro}
can be seen as the sum of an energy-dissipation inequality, with $\E(U)$ replaced with $E$,
with a the weak formulation of~\eqref{eq:general}
and an additional error term 
that is the product of the regularity weight $\mathcal{K}$ with the energy defect, defined as the
difference of the physical energy $\mathcal{E}$ and the auxiliary
variable $E$.
In particular, any weak solution $U$ that satisfies the energy-dissipation inequality
can be identified with an energy-variational solution $(U,E)$ by setting $E=\E(U)$.
This also motivates to choose $\mathcal K$ minimal in a certain sense
since then~\eqref{envarintro} is close to a weak formulation of~\eqref{eq:general}.

Besides the existence results for general systems, 
this novel framework has also been used to identify singular
limits for fluid equations~\cite{envar,eiter2025viscoelastoplasticity}.
or the limit of solutions to an implementable numerical
schemes for the Ericksen--Leslie equations~\cite{Max}. 
Here, the
energy-variational solution for one choice of the regularity weight could
be shown to imply the existence of a measure-valued
solution, whereas another choice provided the convergence of solutions to a
structure-inheriting finite-element scheme.
Energy-variational solutions further come along with a relative-energy inequality,
that can be used to derive a weak-strong uniqueness principle~\cite{envarhyp,viscoenvar}
or to quantify the long-time behavior~\cite{eiterschindler2025timeasymptoticselfsimilarity}
of energy-variational solutions in certain settings.
Moreover, in a recent preprint~\cite{RobVari}, it was shown that
energy-variational framework can also be applied to
the geometric PDE for a 
Navier--Stokes two-phase flow model, and to the system of polyconvex elasticity. 
The energy-variational solutions were shown to be
stronger in comparison to the available varifold solutions
and Young measure-valued solutions, while significantly
reducing the degrees of freedom of the solution concept.

While a generalization of the solvability concept may yield a framework 
that simplifies to show existence results,
it enhances the non-uniqueness of solutions.
Starting with the seminal work by De Lellis and Székelyhidi 
in the context of the incompressible Euler equations~\cite{convexintegration},
the non-uniqueness of solutions was shown for the compressible Euler equations~\cite{Chiodaroli_2014}, the Euler--Korteweg system~\cite{Donatelli_2014}, the
Navier--Stokes equations~\cite{VladTristan} and other systems, using convex integration.
Moreover, the existence of multiple Leray--Hopf solutions to the Navier--Stokes equations
was shown for a particular external forcing~\cite{dallas}, based on a 
linear stability analysis around a special vortex
solution.
This approach was recently adapted for the unforced case
by developing a computer-assisted proof of the relevant
instability~\cite{houwangyang2025nonuniquenesslerayhopfsolutionsunforced}.

In order to counteract this inherent non-uniqueness in fluid dynamics, 
one can reduce the number of solutions by 
introducing suitable selection criteria,
for instance, by physically motivated principles.
This approach goes back to the seminal work by Dafermos~\cite{dafermos1973entropyratecriterion},
on the selection of weak solutions to hyperbolic conservation laws
by maximizing the entropy dissipation rate.
In the context of the Euler equations, 
there are selection criteria that single out solutions with maximal dissipation~\cite{breitfeireislhofmanova2020semiflowcompleteEuler, breitfeireslhofmanova2020semiflowisentropicEuler,maxdiss,selectEuler},
with minimal acceleration~\cite{West}, 
with maximal turbulence~\cite{EmilMaiximalTurb},
or subject to a least-action
principle~\cite{leastaction}.
In~\cite{envar} it was shown that energy-variational solutions
are also suitable to implement certain selection criteria in the context of incompressible fluids.
It was shown that one can select a solution with vanishing energy defect at finitely many prescribed points in time.
Alternatively, the selection by minimization of a strictly convex integral functional is not
only unique but also depends continuously on the data in certain scenarios.

One main novelty of the article at hand is a new existence result 
for energy-variational solutions to the evolution equation~\eqref{eq:general}
by generalizing previous works in several directions. 
We relax the assumptions on the
underlying space as well as on the energy $\E$, 
for which we merely assume a linear growth condition, 
in contrast to superlinear growth assumed in~\cite{viscoenvar}.
Moreover, we reduce the assumptions from~\cite{viscoenvar,envarhyp} 
on the regularity weight $\mathcal{K}$ from convexity to mere
weak* lower semicontinuity. 
This allows for the inclusion of lower order terms in $A$
that can be treated by compactness and strong convergence. 
For the existence proof,
we follow~\cite{envarhyp,viscoenvar,eiter2025viscoelastoplasticity} 
and construct solutions based on a minimizing-movements scheme reminiscent of the
usual Jordan--Kinderlehrer--Otto scheme for gradient flows~\cite{MielkeIntro}.
More precisely, the $n$-th iterate $U^n$ is determined from the previous iterate $U^{n-1}$ by
\begin{equation} \label{eq:discrete.intro}
  \mathcal{E}(U^n)+\tau\Psi^n(U^n)-\mathcal{E}(U^{n-1})-\langle U^n-U^{n-1},\Phi\rangle -\tau \langle A^n(U^n),\Phi \rangle\leq 0 
\end{equation}
for all suitable test functions $\Phi$.
Here $\Psi^n$ and $A^n$ are certain approximations of $\Psi$ and $A$.
The scheme~\eqref{eq:discrete.intro} corresponds to 
a fully implicit discretization of a weak formulation of~\eqref{eq:general}
combined with a discrete energy-dissipation inequality.
We construct the iterate $U^n$ by minimizing the left-hand side of~\eqref{eq:discrete.intro}
after taking the supremum over all admissible $\Phi$.
To ensure a convex-concave structure sufficient for its solvability,
we suitably restrict the class of test functions $\Phi$ in terms of 
another regularity weight $\tilde{\mathcal K}$.
This second regularity weight is only used on the discrete level,
and its effect disappears in the time-continuous limit.
In particular, the choice of $\tilde{\mathcal K}$ can be coarser than $\mathcal K$ in a certain sense
without having an impact on the definition of solutions through~\eqref{envarintro}.

Another main contribution of the present article is the development of different
selection criteria in this general setting of energy-variational solutions. 
We first show that for any countable, ordered sequence of time points in $[0,T]$,
we can construct an energy-variational solution with
vanishing energy defect at those times. 
This is an extension of a similar result from~\cite{envar}.
Moreover, we adapt a criterion from~\cite{selectEuler}, which was used in the context of dissipative solutions to the Euler equations.
We show that for any initial data there exists an energy-variational solution with an arbitrarily small energy defect,
that is, for all $\epsilon > 0$ there exists an energy-variational solution $(U,E)$ such that
\[
\mathcal{E}(U) \leq E \leq \mathcal{E}(U) + \epsilon \quad \text{in } [0,T].
\]
This means in a certain way
that one can find energy-variational solutions that are arbitrarily close 
to a possible weak solution.
As a third criterion, we show that we can select an energy-variational solution that minimizes a given functional among all solutions.
Under additional strict convexity assumptions,
this minimizer is even unique.
A typical example for such a functional would be the $L^1(0,T)$-norm of $E$
in order to minimize the effect of the auxiliary variable
and to maximize the energy dissipation in a certain sense.
Similarly, measure-valued solutions with maximal turbulence
have been identified by maximization of concave functionals in~\cite{EmilMaiximalTurb}.
A selection process by successively minimizing a countable family of convex functionals 
was used in~\cite{breitfeireislhofmanova2020semiflowcompleteEuler,breitfeireslhofmanova2020semiflowisentropicEuler,selectEuler} 
to identify solutions to the Euler equations with maximal dissipation.
Here, the exact energy functionals are not specified
while the selected solution strongly depends on this choice.
In particular, it remains open which choice would lead to a physically reasonable solution.

We exemplify the advantages of the newly developed framework
by concrete PDEs from fluid mechanics. 
At first, we consider the Euler--Korteweg system
for inviscid compressible flow with capillary effects.
The local-in-time existence of strong solutions to this system was established
in~\cite{benzonigavagedanchindescombes2007eulerkorteweg}.
We refer to~\cite{audiardhaspot2017eulerkorteweg} 
for a global-in-time existence result for small data.
Similarly to the Euler equations, 
the Euler--Korteweg equations admit non-unique weak solutions for certain initial data~\cite{Donatelli_2014},
whereas weak solutions coincide with the unique strong solution if the latter exists~\cite{giesselmannlattanziotzavaras2017eulerkorteweg}.
By realizing the Euler--Korteweg system in the setting of energy-variational solutions, this work establishes the first solvability framework
that gives the global-in-time existence for general initial data.
Here, we benefit from the consideration of 
two different regularity weights $\mathcal K$ and $\tilde{\mathcal K}$.
Since the energy functional associated with the Euler--Korteweg system 
contains gradient terms,
the natural \textit{a priori} bounds yield strong compactness 
that can be exploited when passing with the discretized solution to the weak* limit.
In contrast, to ensure existence of the discretized solution constructed by~\eqref{eq:discrete.intro},
the necessary convexity requires a larger regularity weight 
$\tilde{\mathcal K}$.

The second example
illustrates the advantage 
of allowing for a more general functional framework and 
for energies that are homogeneous of degree one.
Such features are usually necessary to treat geometric PDEs, 
and we establish an energy-variational framework for the equations for the binormal curvature flow,
which models the evolution of vortex filaments.
While the quality of the existence results for such geometric PDEs is similar 
to the Euler and Euler--Korteweg equations,
the employed concepts are very different. 
For instance, for the binormal curvature flow, there are results 
on the existence of smooth solutions locally in time~\cite{binormCurveSol},
on the global existence of generalized solutions~\cite[Thm.~1]{BNCF1}, 
and on weak-strong uniqueness~\cite[Thm.~2]{BNCF1},
and the state of the art for mean-curvature flow is similar~\cite{MCF}. 
However, generalized solutions in this area differ
from the concepts in fluid dynamics, as \emph{varifold solutions} are usually considered. 
Roughly speaking, these describe the
limits of evolving manifolds and their tangent spaces.
Therefore, the energy-variational framework presented here may be considered the 
first overarching solution concept that combines the two areas of classical fluid dynamics and geometric PDEs.
Since this allows for the consideration of existence results 
and selection criteria in a unified context,
the energy-variational structure might also be beneficial 
for studying the connection
between fluid flows and the movement of vortex filaments~\cite{BNCF2}
and the related singular limits. 
This, however, remains part of future research.

\paragraph{Outline}
The structure of this article is as follows. 
In Section~\ref{sec:abstract}
we introduce the notion of energy-variational solutions to the 
abstract evolutionary problem~\eqref{eq:general}.
After presenting a number of properties for such solution, 
we provide a general existence result
as well as suitable selection criteria.
In Section~\ref{sec:eulerkorteweg} and Section~\ref{sec:binormal}, 
we study the Euler--Korteweg system and the equation for binormal curvature flow, respectively.
We introduce associated notions of energy-variational solutions
and establish their existence by 
verifying sufficient conditions developed in the abstract framework.

\paragraph{Notation} 
Let $\mathbbm X$ be a Banach space and $\E\colon \mathbbm X\to\R\cup\{\infty\}$ be a functional.
We let $\operatorname{dom}\mathcal{E}:=\{U \in \mathbbm X \;|\;\E(U)<\infty\}$
denote the domain of $\E$. 
If $\E$ is convex, we write $\partial\E$ for its subdifferential with domain $\dom\partial\E$, 
and $\E^*\colon\mathbbm X^*\to\R\cup\{\infty\}$ for the convex conjugate of $\E$. 
The convex indicator function of a set $M \subset \mathbbm X$ is denoted by $\iota_M:\mathbbm X\to \{0,\infty\}$ and satisfies  $\iota_M(x)=0$ if $x \in M$, $\iota_M(x)=+\infty$ if $x \notin M$.

Let $T>0$ and $\mathbbm X$ be the dual of a Banach space.
The space $\mathcal{C}_{w*}([0,T];\mathbbm X)$ contains all weakly* continuous functions $f:[0,T]\to \mathbbm X$.
The space of weakly* measurable functions $f:[0,T]\to \mathbbm X$ that are essentially bounded  
is denoted $L^\infty_{w*}([0,T];\mathbbm X)$.
The total variation of a function $f:[0,T]\to \mathbbm X$ is defined by
\[
|f|_{\mathrm{TV}([0,T];\mathbbm X)}=\sup_{0=t_0< t_1<\ldots<t_N=T}\sum_{i=0}^{N-1}\|f(t_{i+1})-f(t_i)\|_{\mathbbm X}.
\]
The space $\mathrm{BV}([0,T];\mathbbm X)$  of functions of bounded variation
consists of all such functions $f$ such that 
$\|f\|_{\mathrm{BV}([0,T];\mathbbm X)}=\|f\|_{L^\infty_{w*}(0,T;\mathbbm X)} + |f|_{\mathrm{TV}([0,T];\mathbbm X)}< \infty$.
 
The Frobenius product of two matrices $A,B \in \R^{m \times d}$, $d\in\N$, is represented by $A:B:=A_{ij}B_{ij}$, where we implicitly sum over repeated indices.
The symbol $(A)_{\text{sym}}=\frac{1}{2}(A+A^T)$ denotes the symmetric part of $A \in \R^{d\times d}$, and $(A)_{\text{sym},-}$ and $(A)_{\text{sym},+}$ denote the negative semi-definite and positive semi-definite parts of the symmetric matrix $(A)_{\text{sym}}$, respectively. 
By $\mathbbm{I}_d \in \R^{d\times d}$ we denote the identity matrix.
We write $f_+(x)=\max\{f(x),0\}$ and $f_-(x)=\max\{-f(x),0\}$
for the positive and negative part of a function $f:X\to \R$, defined on a set $X$.
Moreover, $\mathbbm 1_M\colon X\to\{0,1\}$ denotes the characteristic function of a set $M\subset X$,
defined by $\mathbbm 1_M(x)=1$ if $x\in M$, and $\mathbbm 1_M(x)=0$ if $x\not\in M$.

If $M\subset\R^d$, $d\in\N$, is open or closed, we write $\C(M)$ and $\C^k(M)$
for the continuous and the $k$-times continuous differentiable functions on $M$ for $k\in\N\cup\{0\}$,
and $\C_c(M)$ and $\C_c^k(M)$ are the respective subspaces of compactly supported functions.
The spaces $\C_0(\R^d)$ and $\C^k_0(\R^d)$ are defined as the closure of $\C^\infty_c(\R^d)$
with respect to the norms
\[
\|f\|_{\C_0(\R^d)}=\sup_{x\in\R^d}|f(x)|,
\qquad
\|f\|_{\C_0^k(\R^d)}=\sum_{|\alpha|\leq k}\|\partial^\alpha f(x)\|_{\C_0(\R^d)},
\]
respectively.
Then the dual space of $\C_0(\R^d)$ can be identified with the 
class of finite Radon measures $\mathcal M(\R^d)$
with the norm given by the total variation
\[
\|\mu\|_{\mathcal M(\R^d)}=\int_{\R^d}\dd{|\mu|}=|\mu|(\R^d),
\]
where the measure $|\mu|$ is the variation of $\mu\in\mathcal M(\R^d)$.
For Lebesgue and Sobolev spaces on a domain $\Omega$, we use the standard notation $L^p(\Omega)$, $W^{k,p}(\Omega)$ and $H^k(\Omega)=W^{k,2}(\Omega)$ 
for $p\in[1,\infty]$ and $k\in\N$.
If $\Omega$ is a bounded domain, we denote the subspaces of $H^1(\Omega)$ and $\C^1(\overline{\Omega})$ 
consisting of elements with vanishing mean by $H_{(0)}^1(\Omega)$ and $\mathcal{C}^1_{(0)}(\overline{\Omega})$. We write $\C(M;\R^d)$, $\C_0(\R^d;\R^d)$, $L^p(\Omega;\R^d)$, \textit{etc.}~for the corresponding classes of vector-valued functions.

\section{Energy-variational solutions in an abstract setting}
\label{sec:abstract}

For the introduction of the concept of energy-variational solutions
to the abstract evolution equation~\eqref{eq:general},
we first collect general assumptions on the relevant functionals and operators.
We derive a series of basic properties of energy-variational solutions.
Subsequently, we address the main results in this framework: 
a general existence result in Subsection~\ref{subsec:existence},
and suitable criteria
to select specific candidates within the solution set
in Subsection~\ref{subsec:selection}.

\subsection{Assumptions}\label{section assump}
In order to carry out the analysis of the abstract evolution equation~\eqref{eq:general}, we introduce the following set of structural and technical assumptions:
\begin{enumerate}[label=({A\arabic*})]
    \item \label{space as}$\mathbbm{Y}, \mathbbm{V}$ are separable Banach spaces 
    with a dense and compact embedding $\mathbbm{Y}\hookrightarrow \mathbbm{V}$.  
    \item\label{energy as} The energy functional $\mathcal{E}:\mathbbm{V}^*\rightarrow [0,\infty]$ is convex and weakly* lower semicontinuous, and it is coercive in the sense that there are $\alpha,\beta\geq 0$ with
    \begin{equation}\label{eq:coercive}
     \forall U \in \V^*: \quad\E(U) \geq \alpha\|U\|_{\mathbbm{V}^*}-\beta.
    \end{equation}
    \item  \label{ass energy bal} 
    The dissipation potential $\Psi:[0,T]\times \mathbbm{V}^*\rightarrow \R \cup \{+\infty\}$ 
    and the operator $A:[0,T]\times (\text{dom}\, \E \cap \text{dom}\, \Psi) \rightarrow \mathbbm{Y}^*$, where we set $\dom \Psi:= \bigcap_{t \in [0,T]}\dom \Psi (t,\cdot)$,  
    are both measurable with respect to the first variable and satisfy the equality
    \begin{equation}\label{eq:dissipation.operator}
    \forall \Phi \in  \Y\cap\dom\partial\E^*\quad \exists \zeta \in \partial\E^*(\Phi)\cap\dom\Psi \quad \forall t \in [0,T]: \quad \langle A(t, \zeta),\Phi \rangle = \Psi(t,\zeta),
    \end{equation}

    \item \label{ass:K.lsc}
    There exists a regularity weight $\mathcal{K}:\mathbbm{Y}\rightarrow [0,\infty)$ that is continuous 
    and maps bounded sets to bounded sets,
    and such that the mapping
    \begin{equation*}
        U \mapsto \Psi(t,U) - \langle A(t,U),\Phi \rangle +\mathcal{K}(\Phi)\mathcal{E}(U)
    \end{equation*}
    is weakly* lower semicontinuous on $\text{dom}\, \E \cap \text{dom}\, \Psi$ for all $\Phi \in \mathbbm{Y}$ and a.a.~$t \in [0,T]$ .
    
    \item \label{ass:K.conv}
    There exists an auxiliary regularity weight $\tilde{\mathcal{K}}:\mathbbm{Y}\rightarrow [0,\infty)$ that is convex, continuous, and maps bounded sets to bounded sets. Moreover, the mapping
    \begin{equation*}
    U\mapsto \Psi(t,U) - \langle A(t,U),\Phi \rangle +\tilde{\mathcal{K}}(\Phi)\mathcal{E}(U)
    \end{equation*}
is convex and weakly* lower semicontinuous on $\text{dom}\, \E \cap \text{dom}\, \Psi$ for all $\Phi \in \mathbbm{Y}$ and a.a.~$t \in [0,T]$. Furthermore, for all $\Phi \in \Y$ there exists a constant $c(\Phi)\geq 0$ such that  
\begin{equation}\label{eq:ass.K.lowerbound}
\Psi(t,U) - \langle A(t,U),\Phi \rangle +\tilde{\mathcal{K}}(\Phi)\mathcal{E}(U) \geq -c(\Phi)(\E(U)+1)
\end{equation}
for all $U \in \V^*$ and a.a.~$t\in(0,T)$. Moreover, for $\Phi \in \Y$ with $\|\Phi\|_\Y\leq R$, where $R>0$, the constant $c(\Phi)$ only depends on $R$ and is idenpendent of $\Phi$.
\end{enumerate}

The space $\Y$ introduced in \ref{space as} serves as the class of test functions, and the dual space $\V^*$ of $\V$ is the state space. Observe that the dense embedding $\Y \hookrightarrow \V$ implies the (dense) embedding $\V^*\hookrightarrow\Y^*$,
and the compactness of one of the embeddings implies 
the compactness of the other by Schauder's theorem.
Moreover, separability ensures that bounded sets in $\V^*$ are sequentially weakly* compact
by the sequential version of the Banach--Alaoglu theorem.

The convexity and the lower semicontinuity from~\ref{energy as}
imply weak* lower semicontinuity of the energy functional $\E$,
and the coercivity assumption ensures that 
the sublevel sets of $\E$ are sequentially weakly* compact by the Banach--Alaoglu theorem. 
This allows to conclude the existence of convergent subsequences from \textit{a priori} bounds.

The assumption~\ref{ass energy bal}
ensures that the dissipation potential $\Psi$ 
is consistent with the operator $A$.
Indeed, consider the regular case where we have the Fenchel equivalence, so that $\zeta\in\partial\E^*(\Phi)$
if and only if $\Phi\in\partial\E(\zeta)$.
Then we can formally test~\eqref{gen eq}
with an element $\Phi(t)\in\partial\E(U(t))$,
and obtain 
\[
0=\big\langle\partial_t U(t) + A(t,U(t)),\Phi(t)\big\rangle 
= \frac{\dd}{\dd t} \E(U(t))+\Psi(t,U(t)),
\]
which gives the energy-dissipation law~\eqref{eq:energy.dissipation.intro}.

Note that for the intersections $\Y\cap\dom\partial\E^*$ and $\partial\E^*(\Phi)\cap\dom\Psi$
we tacitly use the natural embeddings $\V\hookrightarrow\V^{**}$ 
and $\V^*\hookrightarrow\V^{***}$. 
More precisely, $\Phi\in\Y\cap\dom\partial\E^*$ means that 
$\Phi\in\Y\hookrightarrow\V$ and that the 
natural representative of $\Phi$ in $\V^{**}$ is an element of $\dom\partial\E^*$.
Similarly, $\partial\E^*(\Phi)\cap\dom(\Psi)$ has to be understood.

The regularity weights $\mathcal K$ and $\tilde{\mathcal K}$
from~\ref{ass:K.lsc} and~\ref{ass:K.conv}
are fundamental for our approach, but play different roles:
The weak* lower semicontinuity from assumption~\ref{ass:K.lsc} ensures that the solution
concept is stable under weak* convergence.
In particular, it ensures that solutions can be approximated within the weak* topology in $\V^*$,
where convergence can be derived from a piori bounds.
The regularity weight $\tilde{\mathcal K}$ is of a more technical nature. 
The convexity from~\ref{ass:K.conv} ensures that the approximate solutions can be constructed
by solving certain saddle-point problems.
If solutions are approximated differently, for instance, by a space-time discretization
or a vanishing-viscosity approach, the assumption~\ref{ass:K.conv} may be obsolete.

The lower bound~\eqref{eq:ass.K.lowerbound} is needed for technical reasons.
For instance, as $\Psi$ does not have a sign in general, 
it ensures that the energy-dissipation inequality ensures integrability of $\Psi$,
see Proposition~\ref{prop:aprioribounds} below.
Instead of using the energy,
one can also derive an affine lower bound of the left-hand side of~\eqref{eq:ass.K.lowerbound},
which would be natural if $A$ contains terms linear in $\Phi$.
By the coercivity estimate~\eqref{eq:coercive}, such a bound implies~\eqref{eq:ass.K.lowerbound}.

The assumption that $\mathcal K$ and $\tilde{\mathcal K}$ map bounded sets to bounded sets 
is merely technical.
It is trivially satisfied by all examples presented later
since both functionals
are always dominated by a multiple of the norm $\|\cdot\|_{\Y}$ in these cases.

At first sight, the assumption~\ref{ass:K.lsc} seems redundant 
since one can simply choose $\mathcal K=\tilde{\mathcal K}$ if~\ref{ass:K.conv} is satisfied.
However, as $\mathcal K$ appears in the energy-variational formulation~\eqref{envarintro} (compare also Definition~\ref{sol def} below),
it measures to a certain extent how far the energy-variational solution deviates from a weak solution. 
This motivates choosing $\mathcal K$ as small as possible.
The possibility of using two different regularity weights is one of the novelties of this paper,
and it has recently been adapted to treat a visco-elasto-plasticity model~\cite{eiter2025viscoelastoplasticity},
while in other contexts, they coincide~\cite{viscoenvar,envarhyp}.

\subsection{Definition of energy-variational solutions}

Now we can formally introduce energy-variational solutions to
the abstract problem~\eqref{eq:general}.

\begin{definition}\label{sol def}
Let $U_0\in\V^*$ with $\E(U_0)<\infty$.
 The pair $(U,E)$ with
 \[
 U \in L^\infty_{w^*}(0,T;\mathbbm{V}^*) \cap \mathrm{BV}([0,T]; \Y^*),
 \qquad
 E\in\mathrm{BV}([0,T])
 \]
 is called an \textit{energy-variational solution} to \eqref{eq:general} if $\mathcal{E}(U(t)) \leq E(t)$ for almost all  $t \in (0,T)$ and if
\begin{equation}\label{envar def1}
[ E-\langle U, \Phi \rangle ]  \Big|^t_s +\int_s^t \langle U , \partial_t \Phi \rangle- \langle A(\tau,U), \Phi\rangle +\Psi(\tau,U) +\mathcal{K}(\Phi)[\mathcal{E}(U)-E] \dd{\tau} \leq 0
\end{equation}   
for all $\Phi \in \mathcal{C}^1([0,T];\mathbbm{Y}\cap\dom\partial\E^*)$ and almost all $s < t \in [0,T]$, including $s=0$ for which $U(0)=U_0$.
\end{definition}

\begin{remark}
If $(U,E)$ is an energy-variational solution with $E(t)=\E(U(t))$,
then~\eqref{envar def1} simplifies to
\[
\E(U)\Big|^t_s + \int_s^t\Psi(\tau,U)\de\tau
-\langle U, \Phi \rangle \Big|^t_s +\int_s^t \langle U , \partial_t \Phi \rangle- \langle A(\tau,U), \Phi\rangle \de{\tau} \leq 0,
\]  
which is the combination of an energy-dissipation inequality
with a weak formulation of~\eqref{gen eq 1}.
In particular, any weak solution satisfying the energy-dissipation balance
is an energy-variational solution.
\end{remark}

\begin{remark}
    \label{rem:testfunctions}
    In Definition~\ref{sol def}
    we only consider test functions taking values in $\Y\cap\dom\partial\E^*$.
    Here we tacitly use the embedding $\Y\hookrightarrow\V\hookrightarrow\V^{**}$, and
    the intersection has to be understood as in~\eqref{eq:dissipation.operator}.
    We see in Proposition~\ref{prop:Estar} below that
    the domain $\dom \partial\E^*$ contains at least a ball around zero,
    so that the class of test functions is sufficiently rich.
    We further show if $\E$ has superlinear growth, then we even have $\dom\E=\dom \partial\E^*=\V^{**}$, 
    and test functions from the whole linear space 
    $\mathcal{C}^1([0,T];\mathbbm{Y}\cap\dom\partial\E^*)$
    are admissible.
\end{remark}

 \begin{remark}
     Since we do not assume separability of the space $\V^*$, 
     Pettis's theorem that identifies weak and strong measurability (see~\cite[Theorem 1.34]{roubicek2005nonlpde} for instance) does not apply in general.
     Therefore, it is necessary to consider solutions in a Bochner space
     of weakly* measurable functions, namely $L^\infty_{w*}([0,T];\V^*)$
     since it can be identified with the dual space of $L^1([0,T];\V)$ in this general setting,
     compare~\cite[Theorem 6.14]{pedregal1997parametrizedmeasures}.     
 \end{remark}   

\begin{remark}
    Observe that we do not specify an initial value for $E$ in Definition~\ref{sol def},
    which allows for certain flexibility and features like the semiflow property, see Proposition~\ref{prop p iv} below.
    However, in the main existence result from Theorem~\ref{main theorem1},
    we show that the choice $E(0+)=\E(U_0)$ is indeed possible,
    so that the energy defect vanishes initially.
\end{remark}

\subsection{General properties of energy-variational solutions}

We collect a number of properties, which follow from the definition of energy-variational solutions.
Let the assumptions \ref{space as}--\ref{ass:K.conv} be fulfilled for all the following propositions.

First of all, we show that $\dom\partial\E^*$ contains a neighborhood of zero.
In particular, the class of test functions 
in the energy-variational formulation~\eqref{envar def1}
is not too small and
contains at least a ball around zero.

\begin{proposition}
    \label{prop:Estar}
    The convex conjugate $\E^*$ of $\E$ satisfies 
    $\E^*\leq \iota_{\overline B_{\alpha}}+\beta$,
    where $\alpha,\beta\geq 0$ are the constants from~\ref{energy as} and $B_{\alpha}=\{\eta\in \V^{**}\mid \|\eta\|_{\V^{**}}< \alpha\}$.
    In particular, it holds $\overline{B}_{\alpha}\subset\dom\E^*$
    and $B_{\alpha}\subset\dom\partial\E^*$.
    Moreover, if $\E$ has superlinear growth,
    that is, 
    \begin{equation}\label{eq:superlinear}
    \lim_{\|U\|_{V^*}\to\infty}\frac{\E(U)}{\|U\|_{V^*}}=\infty,
    \end{equation}
    then $\dom\E^*=\dom\partial\E^*=\V^{**}$.
\end{proposition}

\begin{proof}
    From the definition of the convex conjugate and from~\ref{energy as}, 
    we infer
    \[
    \E^*(\eta)
    \leq \sup_{U\in\V^*}\langle U,\eta\rangle - \alpha\|U\|_{V^{*}} + \beta
    \leq \sup_{U\in\V^*} \big[\|\eta\|_{V^{**}}-\alpha\big]\|U\|_{V^{*}}+\beta
    =\iota_{B_{\alpha}}(\eta)+\beta.
    \]
    This further implies $\E^*(\eta)<\infty$ for $\eta\in \overline B_{\alpha}$, that is, $\dom\E^*\subset \overline B_{\alpha}$,
    which also yields $\dom\partial\E^*\subset B_{\alpha}$
    as $\dom\partial\E^*$ is dense in $\dom\E^*$ and convex.
    
    Now let $\E$ satisfy~\eqref{eq:superlinear}.
    Then for every $\alpha>0$ there exists $\beta>0$ such that~\eqref{eq:coercive} holds.
    Hence, by the previous argument we have $B_\alpha\subset\dom\partial\E^*\subset\dom\E^*$
    for any $\alpha>0$, so that $\dom\partial\E^*=\dom\E^*=\V^{**}$.
\end{proof}

We can now show that the function class specified in Definition~\ref{sol def} 
is natural for energy-variational solutions, 
and that we can derive associated \textit{a priori} bounds.

\begin{proposition}
    \label{prop:aprioribounds}
    Let $(U,E)$ be an energy-variational solution with $E(0+)\leq M$. 
    Then there exists a constant $C(M)>0$, independent of $(U,E)$, such that
    \begin{equation}
    \norm{U}_{L^\infty_{w^*}(0,T;\mathbbm{V}^*)}
    +\norm{U}_{\mathrm{BV}([0,T]; \Y^*)}
    +\norm{E}_{\mathrm{BV}([0,T])}
    +\norm{\Psi(\cdot,U)}_{L^1(0,T)}
    \leq C(M).
    \end{equation}
\end{proposition}
\begin{proof}
    Plugging $\Phi =0$ into the inequality \eqref{envar def1}, we obtain 
        \begin{align}\label{tgg}
        E\Big|_s^t + \int_s^t \Psi(\tau,U) \dd{\tau} 
        &\leq \int_s^t \mathcal{K}(0)[E(\tau)-\E(U)]\dd{\tau} 
        \leq \int_s^t\mathcal{K}(0)E(\tau)\dd{\tau} 
        \end{align}
       for almost all $0\leq s \leq t \leq T$.
        Since $\Psi(\tau,U) \geq -c(0)(\E(U)+1)-\tilde{\mathcal{K}}(0)\E(U)$ by assumption \ref{ass:K.conv}, inequality \eqref{tgg} this yields 
       \[
        E\Big|_s^t
        \leq \int_s^t [\mathcal{K}(0)+c(0)+\tilde{\mathcal{K}}(0)]E+c(0)\dd{\tau}.
        \]
       We can apply Gronwall's inequality to obtain
$\|E\|_{L^\infty(0,T)}\leq C$. 

Next, we show that $\|E\|_{\mathrm{BV}([0,T])}\leq C$. 
Let us define $h:[0,T]\to \mathbbm{R}$ as 
\[
h(t):=E(t) +\int_0^t \Psi(\tau,U) +\mathcal{K}(0)[\mathcal{E}(U)-E]\dd{\tau}.
\]
Due to~\eqref{tgg},
$h$ is non-increasing, so that $|h|_{\mathrm{TV}([0,T])}=|h(0)-h(T)|\leq h(0)=E(0)$.
We further define $g=h-E$.
On the one hand, with the inequality \eqref{tgg} we estimate $g$ from above as
\begin{align}\label{psi upper}
   g(t)= \int_0^t \Psi(\tau,U) +\mathcal{K}(0)[\E(U) -E]\dd{\tau} \leq -E(t)+E(0)
    \leq 2\|E\|_{L^\infty(0,T)}\leq C.
\end{align}
On the other hand, by assumption \ref{ass:K.conv} we can estimate $\tilde{g}(\tau):= \Psi(\tau,U(\tau)) +\mathcal{K}(0)[\mathcal{E}(U(\tau))-E(\tau)]$ from below by
\[
\begin{aligned}
  \tilde{g}(\tau)&= \Psi(\tau,U(\tau)) +\mathcal{K}(0)[\E(U(\tau)) -E(\tau)]
  \\
  &\geq  -c(0)(\E(U(\tau))+1)+\mathcal{K}(0)[\E(U(\tau))-E(\tau)]-\tilde{\mathcal{K}}(0)\E(U(\tau))
  \geq C_1
   \end{aligned}
\]
due to $\|E\|_{L^\infty(0,T)}\leq C$.
Hence, $\tilde{g} \in L^1(0,T)$, and $g$ is absolutely continuous, and as such $g \in \mathrm{BV}([0,T])$.
We further conclude $\|g\|_{\mathrm{BV}([0,T])}\leq C$, 
and $E=h-g$ implies 
$\|E\|_{\mathrm{BV}([0,T])}\leq C$.
Moreover, the asserted bound on $\Psi(\cdot,U)$ follows from the bounds on $\tilde g$. 

   It remains to show the bounds on $U$.
   From $\E(U(t)) \leq E(t)$, we conclude $ \|\E(U)\|_{L^\infty(0,T)} \leq C$. Combined with coercivity of $\E$ (assumption \ref{energy as}), this implies $\|U\|_{L^\infty_{w^*}(0,T;\V^*)}\leq C$.
    Moreover, rewriting the energy-variational inequality \eqref{envar def1} and adding and subtracting the term $\tilde{\mathcal{K}}(\Phi)\E(U)$, we get 
        \begin{align*}
            -\langle U(t)-U(s),\Phi\rangle
            &\leq -\int_s^t \Psi(\tau,U)-\langle A(\tau,U), \Phi \rangle + \tilde{\mathcal{K}}(\Phi)\E(U)\dd{\tau}\\
            &\qquad
            -(E(t)-E(s))+\int_s^t [\tilde{\mathcal{K}}(\Phi)-\mathcal{K}(\Phi)]\E(U)+\mathcal{K}(\Phi)E \dd{\tau}
        \end{align*}
        for any $\Phi\in\Y\cap\dom\partial\E^*$ and a.a.~$s<t$, $s,t\in[0,T]$.
        By Proposition~\ref{prop:Estar}, this includes all $\Phi\in\Y$ with 
        $\|\Phi\|_{\Y}<\alpha$, where $\alpha\geq0$ is the constant from~\ref{energy as}.
       By assumption \ref{ass:K.conv}, it holds 
        \[
        \inf_{\|\Phi\|_\Y< \alpha}[\Psi(t,U)-\langle A(t,U),\Phi\rangle+\tilde{\mathcal{K}}(\Phi)\E(U)]\geq -c(\E(U)+1)).
         \]
        Using the identity $\|U\|_{\V^*}=\alpha^{-1}\sup_{\|\Phi\|_{\V}<\alpha}\langle U,\Phi\rangle$
        and the assumption that $\mathcal K$ and $\tilde{\mathcal K}$ map bounded sets to bounded sets, 
        we can estimate
        \begin{align*}
            \sup_{0=t_0< t_1<\ldots<t_N=T}\sum_{i=0}^{N-1}\|U(t_{i+1})-U(t_i)\|_{\Y^*}\leq   \int_0^Tc(\E(U)+1) + cE \dd{\tau}+\|E\|_{\mathrm{BV}[0,T]}\leq C.
        \end{align*}
        Together with $\|U\|_{L^\infty_{w^*}(0,T;\V^*)}\leq C$,
        this implies $\|U\|_{\mathrm{BV}([0,T];\Y^*)}\leq C$ and completes the proof.
\end{proof}

Next we show that the energy-variational formulation~\eqref{envar def1},
which holds pointwise a.e. in $[0,T]$,
can be replaced with a pointwise inequality.
If $\E$ has superlinear growth, without loss of generality, 
we may assume that $U$ is weakly* continuous in $\V^*$ (and thus, in $\Y^*$ as well).

\begin{proposition}\label{prop p v}
Let $U \in L^\infty_{w^*}(0,T;\mathbbm{V}^*) \cap \mathrm{BV}([0,T]; \Y^*)$, 
$E\in\mathrm{BV}([0,T])$ and $\Phi \in \mathcal{C}^1([0,T];\mathbbm{Y})$.
Then~\eqref{envar def1} holds for a.a.~$s<t\in[0,T]$ if and only if
\begin{equation}\label{eq:envar.limits.pm}
    [ E-\langle U, \Phi \rangle ]  \Big|^{t+}_{s-} +\int_s^t \langle U , \partial_t \Phi \rangle- \langle A(\tau,U), \Phi\rangle +\Psi(\tau,U) +\mathcal{K}(\Phi)[\mathcal{E}(U)-E] \dd{\tau} \leq 0
\end{equation}
holds for all $s\leq t\in[0,T]$.
Moreover, if $(U,E)$ is an energy-variational solution as in Definition \ref{envar def1},
and if $\E$ satisfies~\eqref{eq:superlinear},
then there exists $\tilde{U}\in \mathcal{C}_{w*}([0,T];\V^*)$ such that $U=\tilde{U}$ almost everywhere on $[0,T]$.
\end{proposition}

\begin{proof}
The assumptions on $U$, $E$, $\Phi$ imply
$[t \mapsto E(t) - \langle U(t), \Phi(t) \rangle] \in \mathrm{BV}([0,T])$.
One-sided limits of this mapping exist. 
On the one hand, passing to such limits in~\eqref{envar def1} yields~\eqref{eq:envar.limits.pm}.
On the other hand, as these limits coincide with the function values a.e.,~\eqref{eq:envar.limits.pm} implies~\eqref{envar def1}.

In particular, if $(U,E)$ is an energy-variational solution, it satisfies~\eqref{eq:envar.limits.pm}.
Assuming superlinear growth of $\E$, we have $\dom\E^*=\dom\partial\E^*=\V^{**}$ by Proposition~\ref{prop:Estar},
so that the we can take arbitrary test functions in $\Phi\in \C^1([0,T];\Y)$ in~\eqref{eq:envar.limits.pm}.
Choosing $t=s$ and $\Phi \in \Y$ independent of time, yields 
\[
[ E-\langle U, \Phi \rangle ]  \Big|^{t+}_{t-} \leq 0 \quad \text{for all} \; t \in (0,T).
\]
From a standard variational argument (compare~\cite[Lemma 2.10]{envarhyp}),
it follows that $U(t+) = U(t-)$ in \( \Y^* \) for all \( t \in (0,T) \).
Therefore, the function $\tilde U$ defined by $\tilde U(t):=U(t+)$ in $\Y^*$
satisfies $\tilde{U}\in \mathcal{C}_{w^*}([0,T]; \Y^*) $,
and as $U\in\BV([0,T];\Y^*)$, we have that
$U=\tilde U$ a.e.~in $[0,T]$.
Moreover, for $\{t_n\}\subset[0,T]$ with $t_n \to t\in[0,T]$, we have 
$\|\tilde{U}(t_n)\|_{\V^*}\leq \| U\|_{L^\infty(0,T;\V^*)}$, so that 
there exists a (not relabeled) subsequence with
$\tilde{U}(t_n) \stackrel{*}{\rightharpoonup} \tilde{U}_{t}$ in $\V^*$
for some element $\tilde{U}_{t}\in\V^*$.
Since we also have $\tilde{U}(t_n) \stackrel{*}{\rightharpoonup} \tilde{U}(t)$ in $\mathbbm{Y}^* $, we infer $\tilde{U}_{t}=\tilde{U}(t)$ due to the dense embedding $\Y \hookrightarrow\V$,
so that $\tilde{U}(t_n)\to \tilde U(t)$ in $\V^*$. In conclusion, we obtain $\tilde{U}\in \mathcal{C}_{w^*}([0,T];\V^*) $.
\end{proof}

The following result shows that one-sided limits of $U$ exist in the weak* topology of $\V^*$.
Moreover, a vanishing energy defect at a certain time
yields right-continuity of the energy in this point.

\begin{proposition}\label{prop point ii}
At each point $t_0\in[0,T]$, 
the one-sided weak* limits 
$U(t_0+):=\lim_{t \searrow t_0}U(t)$ and $U(t_0-):=\lim_{t \nearrow t_0}U(t)$
exist in $\V^*$.
Moreover, if $t_0 \in [0,T]$ such that $\E(U(t_0+))=E(t_0+)$, then $\lim_{t \searrow t_0} \E(U(t))=\E(U(t_0+))$.  
\end{proposition}
\begin{proof}
        Firstly, since $U\in\BV([0,T];\Y^*)$, the one-sided limit $U(t_0+)$ and $U(t_0-)$ 
        exists in $\Y^*$,
        and by $U\in L^\infty(0,T;\V^*)$, the same argument as in the proof of Proposition~\ref{prop p v} yields their existence in $\V^*$.
        
        As $E\geq \E(U)$ and as $\E$ is weakly* lower semicontinuous, we further observe that 
        \begin{align*}
            E(t_{0}+)= \lim_{t\searrow t_0}E(t)
             \geq\limsup_{t\searrow t_0} \E(U(t))
            \geq\liminf_{t\searrow t_0} \E(U(t)) 
            \geq \E(\lim_{t\searrow t_0} U(t)) 
            \geq  \E(U(t_{0}+)).
        \end{align*}
        If the first and the last term in this chain of inequalities coincide,
        we conclude that $\E(U(t))\to\E(U(t_0+))$ as $t\searrow t_0$.
\end{proof}

Another useful characteristic of energy-variational solutions
is the following semi-flow property.

\begin{proposition}\label{prop p iv}
  Energy-variational solutions possess the semi-flow property:
  If $0<t_0<T$ and $(U_1,E_1)$ is a solution on $[0,t_0]$ and $(U_2,E_2) $ is a solution on $[t_0,T]$
  such that $U_1(t_0-)=U_2(t_0+)$ in $\Y^*$ and $E_2(t_0+) \leq E_1(t_0-)$, then the concatenation 
    \begin{equation*}
    (U(t),E(t)):=\begin{cases}
    (U_1(t),E_1(t)) \quad \text{if } t \in [0,t_0), \\
    (U_2(t),E_2(t)) \quad \text{if } t \in [t_0,T] ,
    \end{cases}
\end{equation*}
is a solution. If $(U,E)$ is an energy-variational solution on $[0,T]$, its restriction 
to a subinterval $[t_0,t_1]\subset[0,T]$ is an energy-variational solution on $[t_0,t_1]$
with initial value $U(t_0+)$.
\end{proposition}
\begin{proof}
   Let $0<t_0<T$ and let $(U,E)$ be defined as above. Then $(U,E)$ satisfies the variational inequality \eqref{envar def1}. Indeed, since the variational inequality \eqref{envar def1} is formulated in an integral form, we decompose the integral over $[0,T]$ as two integrals over $(0,t_0)$ and over $(t_0,T)$. On the interval $[0,t_0)$ the inequality holds for $(U_1,E_1)$, and on $[t_0,T]$ it holds for $(U_2,E_2)$ by assumption, with test functions that are suitably restricted.
   Therefore, the variational inequality \eqref{envar def1} also holds for the concatenated pair $(U,E)$ on $[0,T]$. Moreover, due to weak* lower semicontinuity of $\E$, it holds
            \[
        \mathcal{E}(U(t_0+)) = \mathcal{E}(U_2(t_0+)) \leq \liminf_{t_n \searrow t_0} \mathcal{E}(U_2(t_n)) \leq \liminf_{t_n \searrow t_0} E_2(t_n) = E_2(t_0+) = E(t_0+).
            \]
            Thus the concatenation $(U,E)$ is also an energy-variational solution in the sense of the Definition \ref{sol def}.

            Furthermore, if $(U,E)$ is a solution on $[0,T]$, then  its restriction to any subinterval satisfies the same energy-variational inequality on any interval $[s,t] \subset [0,T]$. 
\end{proof}

An advantage of energy-variational solutions compared to weak solutions
is the structure of the solution set. 
As the first feature, we show its convexity in the case that $\mathcal K$ and $\tilde{\mathcal K}$ coincide.

\begin{proposition}\label{prop p iii}
Let $M\in[\E(U_0),\infty]$. 
If the regularity weights $\mathcal{K}, \tilde{\mathcal{K}}$ are such that $\mathcal{K}=\tilde{\mathcal{K}}$, then the solution set $\mathcal{S}(U_0,M)$ defined by
\begin{equation}
\label{eq:solset}
    \mathcal{S}(U_0,M):=\big\{(U,E) \;\text{is an energy-variational solution with initial data}\; U_0 \text{ and }E(0+)\leq M\big\}
\end{equation}
is convex.
\end{proposition}
\begin{proof}
The conditions $\E(U)\leq E$ and $E(0+)\leq M$ are preserved under convex combinations,
and if $\mathcal K=\tilde{\mathcal K}$, then
the left-hand side of energy-variational formulation~\eqref{envar def1} 
is convex as a function of $(U,E)$ by~\ref{ass:K.conv}.
This shows convexity of $\mathcal{S}(U_0,M)$.
\end{proof}

We next show the sequential weak* compactness of the solution set.
For this property, the weak* lower semicontinuity induced by 
the regularity weight $\mathcal K$ from~\ref{ass:K.lsc} is crucial.

\begin{proposition}\label{prop p i}
Let $U_0\in\V^*$ and $M\in[\E(U_0),\infty)$.
Then the set $\mathcal{S}(U_0,M)$, defined in~\eqref{eq:solset}, is sequentially weakly* compact in the space $L^\infty_{w^*}(0,T;\V^*) \cap \mathrm{BV}([0,T];\Y^*)\times\BV([0,T])$.
\end{proposition}

\begin{proof}
        Let $\{(U^n,E^n)\}_{n\in\N} \subset \mathcal{S}(U_0,M)$. 
        From Proposition~\ref{prop:aprioribounds}
        we obtain the uniform bounds
        \[\|U^n\|_{L^\infty_{w^*}(0,T;\V^*) \cap \mathrm{BV}([0,T];\Y^*)}\leq c, \quad \|E^n\|_{\mathrm{BV}([0,T])}\leq c\]
        for some $c>0$.
        From the sequential Banach--Alaoglu theorem we infer that there exist subsequences (not relabeled), such that $U^n \stackrel{*}{\rightharpoonup}U$ in $L^\infty_{w^*}(0,T;\V^*) \cap \mathrm{BV}([0,T];\Y^*)$ and $E^n \stackrel{*}{\rightharpoonup} E $ in $\mathrm{BV}([0,T])$. For each $n \in \mathbbm{N}, s < t \in [0,T], \Phi \in \mathcal{C}^1([0,T];\Y)$ it holds
\begin{equation}
\label{eq:envar.seq}
[ E^n-\langle U^n, \Phi \rangle ]  \Big|^t_s +\int_s^t \langle U^n , \partial_t \Phi \rangle- \langle A(\tau,U^n), \Phi\rangle +\Psi(\tau,U^n) +\mathcal{K}(\Phi)[\mathcal{E}(U^n)-E^n] \dd{\tau} \leq 0.
\end{equation}   
Due to a generalization of Helly's selection theorem
(see~\cite[Theorem B.5.10]{mielkeroubicek2015ris}), we have $U^n(t)\stackrel{*}{\rightharpoonup} U(t)$ in $\Y^*$
and $E^n(t) \rightarrow E(t)$ a.e.~on $[0,T]$. 
Arguing similarly as in the proof of Proposition~\ref{prop p v},
we conclude $U^n(t)\rightharpoonup U(t)$ in $\V^*$.
Moreover, from the compact embedding $\mathrm{BV}([0,T]) \hookrightarrow L^1(0,T)$, the convergence $E^n \rightarrow E$ in $L^1(0,T)$ follows. 
These convergence suffice to pass to the limit in all terms in~\eqref{eq:envar.seq}
where $(U^n,E^n)$ appears in a linear way.
For the remaining terms, 
we use Fatou's lemma together with 
the lower semicontinuity property from~\ref{ass:K.lsc}
to conclude
\[
\begin{aligned}
\liminf_{n\to\infty}
\int_s^t - \langle A(\tau,U^n), \Phi\rangle +\Psi(\tau,U^n) +\mathcal{K}(\Phi)\mathcal{E}(U^n)\dd{\tau}
&\geq 
\int_s^t - \langle A(\tau,U), \Phi\rangle +\Psi(\tau,U) +\mathcal{K}(\Phi)\mathcal{E}(U)\dd{\tau}.
\end{aligned}
\]
Note that the application of Fatou's lemma was justified
as~\eqref{eq:ass.K.lowerbound} induces an integrable lower bound 
for the integrand,
compare also estimate~\eqref{eq:Fatou.lowerbound} in the proof of Theorem~\ref{main theorem1} below.
In conclusion, 
we can thus pass to the limit inferior on the left-hand side of~\eqref{eq:envar.seq},
which yields~\eqref{envar def1}. 
Moreover, passing to the limit inferior in the condition $\E(U^n)\leq E^n$ a.e.~in $[0,T]$
yields $\E(U^n)\leq E^n$ a.e.~in $[0,T]$.
Thus the limiting pair $(U,E)$ is an energy-variational solution in the sense of Definition \ref{sol def}.
\end{proof}

\subsection{Existence of energy-variational solutions}
\label{subsec:existence}

We now come to the first main result of this article,
which states the global-in-time existence of energy-variational solutions 
to the abstract problem~\eqref{eq:general}.

\begin{theorem}[Existence of energy-variational solutions]\label{main theorem1}
Let the assumptions  \ref{space as}-\ref{ass:K.conv} be satisfied. For every $U_0 \in \operatorname{dom}\E$, there exists an energy-variational solution to \eqref{eq:general} in the sense of Definition \ref{sol def}, which satisfies $E(0+)=\mathcal{E}(U_{0})$.
\end{theorem}

The proof of Theorem~\ref{main theorem1} is divided into two parts. Firstly, we establish the existence of solutions at the time-discrete level. Secondly, we show convergence of these discretized solutions to an energy-variational solution.

\subsubsection{Construction of approximate solutions}

As the first step towards a proof of Theorem~\ref{main theorem1}, we construct approximate solutions by a time-discrete scheme,
consisting of solving a saddle-point problem in each time step.

 Fix $N \in \mathbbm{N}$ such that $N >\tilde{\mathcal{K}}(0)$ and set $\tau = T/N$. In order to obtain an equidistant discretization, we set $t^n = \tau n$ for $n \in  \{0,1,\ldots, N\}$. We define  the class of test functions as 
 \[
 \mathbbm{Y}_\tau:=\{\Phi\in \mathbbm{Y}\cap\dom\partial\E^* \; |\; \tilde{\mathcal{K}}
(\Phi)\leq 1/\tau\},
\]
and we set
    \[
    \begin{aligned}
         A^n(U)&:= \frac{1}{\tau}\int_{t^{n-1}}^{t^n} A(t,U) \dd{t},
         &\qquad
        \Psi^n(U)&:= \frac{1}{\tau}\int_{t^{n-1}}^{t^n} \Psi(t,U) \dd{t}, 
    \end{aligned}
    \]
     for $U \in \mathbbm{V}^*$. We define the function $\mathcal{F}^n:\mathbbm{V}^*\times \mathbbm{Y}_\tau\rightarrow (-\infty,\infty]$ as
\begin{align} \label{function F1}
  \mathcal{F}^n(U|\Phi)&=(\mathcal{E}(U)+\tau\Psi^n(U)-\mathcal{E}(U^{n-1}))-\langle U-U^{n-1},\Phi\rangle -\tau \langle A^n(U),\Phi \rangle, 
\end{align}
where we set $\mathcal{F}^n(U|\Phi)=\infty$ if $U \notin \dom \E$.

We defined $U^n \in \V^*$, $n \in \{0,1,\ldots, N\}$ iteratively by the following minimization scheme:
    \begin{align}
    \begin{split}
        U^0&=U_0, \\
        U^n &= \argmin_{U \in \V^*} \sup_{\Phi \in \Y_\tau} \mathcal{F}^n(U|\Phi). \label{min problem}
        \end{split}
    \end{align}
The proof of existence of solutions $U^n$ is preceded by the following auxilliary lemma.

\begin{lemma}[$\mathcal{H}$ is proper]\label{H is proper}
  Let $\mathcal{H}^n:\V^* \rightarrow \R \cup \{+\infty\}$ be defined as 
  \[
   \mathcal{H}^n(U):= \sup_{\Phi \in \Y_\tau} \mathcal{F}^n(U | \Phi).
  \]
  Then $\mathcal H$ is a convex, weakly* lower semicontinuous and proper functional such that 
  \[
  \inf_{U \in \V^*}\mathcal{H}^n(U) \leq 0.
  \]
\end{lemma}

\begin{proof}
Let $\Phi \in \Y_\tau$ be arbitrary.
By the energy-balance assumption \ref{ass energy bal},
there exists $\hat{U}\in\partial \E^*(\Phi)\cap\dom\Psi^n$
such that $\langle A^n(\hat U),\Phi \rangle = \Psi^n(\hat U)$. Then it holds
\[
    \inf_{U \in \V^*} \mathcal{F}^n(U|\Phi) 
    \leq \mathcal{F}^n(\hat{U}|\Phi) 
    =\mathcal{E}(\hat{U})-\mathcal{E}(U^{n-1})-\langle \hat{U}-U^{n-1},\Phi\rangle \leq 0.
\]
Since $\Phi$ was arbitrary,  
the claimed inequality follows if
\begin{align}\label{minmax}
  \inf_{U\in \V^*}\sup_{\Phi \in \Y_\tau} \mathcal{F}^n(U|\Phi)=\sup_{\Phi \in \Y_\tau}\inf_{U\in \V^* }\mathcal{F}^n(U|\Phi)  .
\end{align}  
We observe that the function $\mathcal{F}^n(U|\Phi)$ is concave with respect to $\Phi$, since it is linear in $\Phi$ and defined on the convex subset $\Y_\tau$. Moreover, it is continuous with respect to $\Phi$. To show convexity in $U$, we rewrite $\mathcal{F}^n$ as
    \begin{align}
            \mathcal{F}^n(U|\Phi)&=(1-\tau \K )[\mathcal{E}(U)-\mathcal{E}(U^{n-1})]-\langle U-U^{n-1},\Phi\rangle \\
            &\phantom{=}+\tau\Psi^n(U)- \tau \langle A^n(U),\Phi \rangle +\K[\mathcal{E}(U)-\E(U^{n-1})]  .\nonumber  
\end{align}
Since $\E$ is convex, since $1-\tau \K\geq 0$, and since the second line in the above expression is convex thanks to the assumption \ref{ass:K.conv}, we conclude that $\mathcal{F}^n$ is convex with respect to $U$. Similarly, we see that $\mathcal{F}^n$ is weakly* lower semicontinuous with respect to $U$.
In particular, this implies that $\mathcal H$ is convex and weakly* lower semicontinuous
as the supremum of functionals with these properties.
Moreover, we conclude that sublevel sets of the function $\mathcal{F}^n(\cdot|0)$ are weakly* closed.
They are also bounded 
due to the coercivity assumption on $\E$ from~\ref{energy as}. 
Therefore, we can apply~\cite[Theorem 2.130]{barbu}, and we conclude the equality \eqref{minmax},
which finishes the proof.
\end{proof}

We can now show the existence of solutions to the minimization problems from~\eqref{min problem}.

\begin{lemma}[Existence of a discrete minimizer]\label{ex min1}
    Let $N\in\N$ with $\tilde{\mathcal K}(0)<N$.
    Then there exists a solution $(U^0,\dots,U^N)$ to the minimization problem~\eqref{min problem},
    Moreover, for $n=1,\dots,N$, the minimizer $U^n$ satisfies 
    \begin{equation}
    \label{eq:Fn.leq0}
    \sup_{\Phi \in \Y_\tau}\mathcal{F}^n(U^n|\Phi) \leq 0.
    \end{equation}
    \end{lemma}

\begin{proof} 
    The minimization problem \eqref{min problem} is equivalent to finding a unique minimizer $U^n$ of the functional $\mathcal{H}^n$ from Lemma \ref{H is proper} for $n=1,\dots,N$.
    Moreover, these minimizers directly satisfy~\eqref{eq:Fn.leq0}.
    Since $\mathcal{H}$ is convex and weakly* lower semicontinuous by Lemma~\ref{H is proper},
    level sets of $H$ are sequentially closed in the weak* topology. Moreover, since $\mathcal{H}$ is coercive,  which follows from~\ref{energy as},
    level sets are also bounded in $\V^*$. 
    The Banach--Alaoglu Theorem implies that level sets are sequentially compact with respect to the weak* topology.
    Since $\mathcal{H}$ is also convex by Lemma~\ref{H is proper},    
    the existence of a solution to the minimization problem \eqref{min problem}
    follows from~\cite[Prop. 1.48]{relax}. 
\end{proof}

From the solution $(U^0,\dots,U^N)$
to the minimization problem~\eqref{min problem},
we now introduce approximate solutions for the energy-variational solution $(U,E)$.
We define piecewise constant prolongations $\ol{U}^N\colon[0,T]\to\V^*$ 
together with the corresponding energy $\ol E^N\colon[0,T]\to[0,\infty)$ as
\begin{equation}\label{eq:prolongations}
    \ol{U}^N(t)=\begin{cases}
    U^0=U_0 \quad &\text{for} \quad t=0, \\
    U^n \quad &\text{for} \quad t \in (t^{n-1},t^n] ,
    \end{cases}
    \qquad\quad
    \ol E^N(t)=\begin{cases}
    \mathcal{E}(U_0) \quad &\text{for} \quad t=0, \\
    \mathcal{E}(U^n) \quad &\text{for} \quad t \in (t^{n-1},t^n] .
    \end{cases}
\end{equation}

To obtain a variational problem satisfied by $(\ol{U}^N,\ol E^N)$,
we introduce piecewise constant and piecewise linear approximations of test functions.
For functions $\chi \in \mathcal{C}([0,T];X)$ where $X$ is $\R$ or $\Y$, let us define
\begin{align*}
    \overline{\chi}^N(t)= \begin{cases}
    \chi(0) \quad &\text{for} \quad t=0 \\
    \chi(t^n) \quad &\text{for} \quad t \in (t^{n-1},t^n] ,
    \end{cases}
    \qquad\quad
    \underline{\chi}^N(t)= \begin{cases}
    \chi(t^{n-1}) \quad &\text{for} \quad t \in [t^{n-1},t^n
    ),\\
    \chi(T) \quad &\text{for} \quad t=T ,
      \end{cases} 
\end{align*}
and
\[
    \hat{\chi}^N(t)= \frac{\chi(t^n)-\chi(t^{n-1})}{\tau}(t-t^{n-1})+\chi(t^{n-1}) \quad \text{for} \quad t \in [t^{n-1},t^n],
\]
for all $n \in \{1,\ldots,N\}$.

\begin{lemma}
For each $N\in\N$, the approximate solution $(\ol{U}^N,\ol E^N)$ satisfies
\begin{equation}\label{eq:envar.approximate}
\begin{aligned}
    \int_0^T -\partial_t \hat{\phi}^N \Big[\ol E^N - \langle \ol{U}^N, \underline{\Phi}^N\rangle \Big] 
    +\overline{\phi}^N\langle \ol{U}^N,\partial_t \hat{\Phi}^N\rangle
    +\underline{\phi}^N\Big[\Psi(t,\ol{U}^N) &-\langle A(t,\ol{U}^N), \underline{\Phi}^N\rangle\Big]\dd{t} \\  
    &\leq\phi(0)\big[\mathcal{E}(U_0)-\langle U_0, \Phi(0)\rangle\big]
\end{aligned}
\end{equation}
for all $\Phi \in \mathcal{C}^1([0,T];\Y_\tau)$ 
and
$\phi \in \mathcal{C}^1_c([0,T))$ such that $\phi \geq 0$.
\end{lemma}

\begin{proof}
By Lemma \ref{ex min1}, we have the inequality
    \begin{equation}\label{disc var1}
  \mathcal{F}(U^n,\Phi)= \mathcal{E}(U^n)+\tau \Psi^n(U^n)-\mathcal{E}(U^{n-1})-\langle U^n-U^{n-1},\Phi \rangle
    -\tau \langle A^n
    (U^n),\Phi \rangle \leq 0 
    \end{equation}
    for all $\Phi\in \Y_\tau$. 
For  $\Phi \in \mathcal{C}^1([0,T];\Y_\tau)$ 
and $\phi \in \mathcal{C}^1_c([0,T))$ with $\phi\geq 0$,
we define $\phi^n:=\phi(t^n)$ and $\Phi^n=\Phi(t^n)$. 
In order to obtain~\eqref{eq:envar.approximate} from the discrete formulation~\eqref{disc var1} for each $n \in \{1,\ldots,N\}$, we use $\Phi=\Phi^{n-1}$ in~\eqref{disc var1}, 
multiply with $\phi^{n-1}\geq 0$, and sum over $n=0,\dots,N$ to obtain 
\begin{align}\label{discrete ineq}
 \sum_{n=1}^N \Big[\phi^{n-1} &\left[E^n-E^{n-1}+\tau\Psi^n(U^n)\right] - \phi^{n-1}\langle U^n-U^{n-1},\Phi^{n-1} \rangle -\tau\phi^{n-1}\langle A^n(U^n),\Phi^{n-1} \rangle  \Big]\leq 0 .
\end{align}
Performing discrete integration by parts and dividing by $\tau >0$, we obtain
\[
\begin{aligned}
    -\sum_{n=1}^N \left(\frac{\phi^n-\phi^{n-1}}\tau \right)E^n +\phi^{n-1}\Psi^n(U^n)+\frac{\phi^n-\phi^{n-1}}{\tau}\langle U^n,\Phi^{n-1}\rangle  -\phi^{n} \langle U^n,\frac{ \Phi^n-\Phi^{n-1}}{\tau}\rangle& \\
    -\sum_{n=1}^N \phi^{n-1} \langle A^n(U^n),\Phi^{n-1} \rangle  -\phi^0\left(E^0-\langle U^0,\Phi^0\rangle \right)  &\leq 0, 
\end{aligned}
\]
where we used that $\phi^N=\phi(T)=0$. 
Inserting the definition of the prolongations, we arrive at~\eqref{eq:envar.approximate}.
\end{proof}

\subsubsection{Convergence of approximate solutions}

We now derive convergence properties
for the approximate solutions $\{(\ol{U}^N, E^N)\}_{N \in \N}$,
which will finally allow us to pass to the limit in~\eqref{eq:envar.approximate}.

\begin{lemma}[Convergence of prolongations]
There exist elements
\[
 U \in L^\infty_{w^*}(0,T;\mathbbm{V}^*) \cap \mathrm{BV}([0,T]; \Y^*),
 \qquad
 E\in\mathrm{BV}([0,T])
 \]
such that the sequence $\{(\ol{U}^N, \ol E^N)\}_{N \in \N}$ of approximate solutions, defined in~\eqref{eq:prolongations},
satisfies the following convergences (up to a subsequence):
\begin{alignat}{2}
    \ol{U}^N &\stackrel{*}{\rightharpoonup} U  \quad &&\text{in } L^\infty_{w^*}(0,T;\mathbbm{V}^*)\cap\BV([0,T];\Y^*), \label{conv u}\\
    \ol{U}^N(t) &\stackrel{*}{\rightharpoonup} U(t) \quad &&\text{in }\V^*\text{ for all~}t\in[0,T], 
    \label{eq:conv.U.strong}\\
      \ol E^N &\stackrel{*}{\rightharpoonup}E \quad &&\text{in } \mathrm{BV}([0,T]), \label{conv e}\\
    \ol {E}^N &\rightarrow {E} \quad &&\text{strongly in } L^1(0,T),
    \label{eq:conv.E.strong}
    \\
    \ol {E}^N(t) &\rightarrow {E}(t) \quad &&\text{for all } t\in[0,T].
    \label{eq:conv.E.ptw}
\end{alignat}
\end{lemma}

\begin{proof} 
To first derive suitable bounds on $(\ol{U}^N,\ol E^N)$,
we make use of the discrete inequality~\eqref{discrete ineq}
derived in the previous proof.
Let $k,\ell\in\{0,\dots,N\}$ with $k<\ell$. 
Using $\phi$ with $\phi^{n-1}= 1$ for $k+1\leq n\leq \ell$
and $\phi^{n-1}= 0$ else, we obtain 
\[
 E^\ell-E^{k}+\sum_{n=k+1}^\ell \Big[\tau\Psi^n(U^n) - \langle U^n-U^{n-1},\Phi^{n-1} \rangle -\tau\langle A^n(U^n),\Phi^{n-1} \rangle  \Big]\leq 0 .
\]
Starting from this inequality, we can now proceed in the same way as in the proof
of Proposition~\ref{prop:aprioribounds}
to derive a uniform bound of the form 
$\|\ol E^N\|_{\BV([0,T])}+\|\ol{U}^N\|_{L^\infty_{w^*}(0,T;\V^*)}+\|\ol{U}^N\|_{\BV([0,T];\Y^*)}\leq C$.
By the sequential form of the Banach--Alaoglu theorem,
we conclude the existence of limit functions $U$ and $E$
such that, up to a subsequence,~\eqref{conv u} and~\eqref{conv e} follow.
The compact embedding $\mathrm{BV}([0,T]) \hookrightarrow L^1(0,T)$ (e.g., see~\cite[Theorem 5.5]{evansgariepy}) 
provides the convergence~\eqref{eq:conv.E.strong},
and~\eqref{eq:conv.E.ptw} follows from Helly's selection theorem.
To show~\eqref{eq:conv.U.strong},
we use the generalized Helly's selection theorem 
from~\cite[Theorem B.5.10]{mielkeroubicek2015ris}
to obtain a (not relabeled) subsequence such that 
$\ol{U}^N(t)\stackrel{*}{\rightharpoonup}U(t)$ in $\Y^*$ for all $t\in[0,T]$.
Combining this convergence with the boundedness of $\{\ol{U}^N\}_{N\in\N}$
in $L^\infty_{w^*}(0,T;\V^*)$,
we can argue as in the proof of Proposition~\ref{prop p v}
that we even have $\ol{U}^N(t)\stackrel{*}{\rightharpoonup}U(t)$ in $\V^*$,
which is~\eqref{eq:conv.U.strong}.
\end{proof}

As the final step, we pass to the limit $N\to\infty$ in~\eqref{eq:envar.approximate}
and show
that $(U,E)$ is an energy-variational solution.
The lower semicontinuity of the regularity weight $\mathcal K$ from~\ref{ass:K.lsc}
ensures that this limit passage towards the formulation~\eqref{envar def1} is possible.

\begin{proof}[Proof of Theorem~\ref{main theorem1}]
First of all, note that by the weak* lower semicontinuity of $\E$ 
and the convergence properties~\eqref{eq:conv.E.strong} 
and~\eqref{eq:conv.U.strong}, we have
$\E(U(t))\leq E(t)$ for all $t\in[0,T]$.
Moreover, one readily shows that the test functions satisfy
\begin{equation}\label{eq:conv.testfct}
\begin{aligned}
 \partial_t \hat{\Phi}^N &\rightarrow \partial_t \Phi, \quad &
 \underline{\Phi}^N &\rightarrow \Phi , \quad &
 \nabla \underline{\Phi}^N &\rightarrow \nabla \Phi \quad &
 \text{in }\mathbbm{Y}
\text{ uniformly in }[0,T], \\
\partial_t \hat{\phi}^N &\rightarrow \partial_t \phi, \quad &
\overline{\phi}^N &\rightarrow \phi, \quad &
\underline{\phi}^N&\rightarrow \phi \quad &
\text{uniformly in } [0,T],
\end{aligned}
\end{equation}
as $N\to\infty$.
Fix $\Phi \in \mathcal{C}^1([0,T];\Y\cap\dom\partial\E^*)$ 
and
$\phi \in \mathcal{C}^1_c([0,T))$ such that $\phi \geq 0$.
Then $\Phi \in \mathcal{C}^1([0,T];\Y_\tau)$ for all $\tau>0$ sufficiently large.
To pass to the limit inferior in~\eqref{eq:envar.approximate},
we examine each term in the inequality separately, according to its structure. For the terms where $(\ol{U}^N, \ol E^N)$ appears linearly, we can directly pass to the limit by~\eqref{conv u},~\eqref{conv e},~\eqref{eq:conv.E.strong} and~\eqref{eq:conv.testfct} to obtain
\[
    \lim_{N \rightarrow \infty} \int_0^T -\partial_t \hat{\phi}^N\left[ \ol E^N -\langle \ol{U}^N, \underline{\Phi}^N\rangle\right] +\overline{\phi}^N\langle \ol{U}^N,\partial_t \hat{\Phi}^N\rangle  \dd{t} 
    =\int_0^T - \partial_t \phi \left[E- \langle  U, \Phi\rangle \right] + \phi \langle U, \partial_t \Phi \rangle \dd{t} .
\]
For the remaining $N$-dependent terms on the left-hand side of~\eqref{eq:envar.approximate}, 
we make use of the lower semicontinuity induced by the regularity weight $\mathcal K$
from~\ref{ass:K.lsc}. We first write
\begin{align*}
    &\int_0^T \underline{\phi}^N\Big({\Psi}(t,\ol{U}^N) -\langle A(t,\ol{U}^N), \underline{\Phi}^N\rangle \Big) \dd{t}
    \\
    &\quad =\int_0^T \underline{\phi}^N\Big({\Psi}(t,\ol{U}^N) -\langle A(t,\ol{U}^N), \underline{\Phi}^N\rangle + \mathcal{K}(\underline{\Phi}^N) \E(\ol{U}^N) \Big) \dd{t}- \int_0^T\underline{\phi}^N\mathcal{K}(\underline{\Phi}^N)\ol E^N \dd{t} .
\end{align*}
The last term is linear in $\ol E^N$ and we can pass to the limit by~\eqref{eq:conv.E.strong} using the continuity of $\mathcal K$.
For the other integral, observe that the integrand has an integrable lower bound since
\begin{equation}
\label{eq:Fatou.lowerbound}
\begin{aligned}
    &\underline{\phi}^N\Big({\Psi}(t,\ol{U}^N) -\langle A(t,\ol{U}^N), \underline{\Phi}^N\rangle + \mathcal{K}(\underline{\Phi}^N) \E(\ol{U}^N)\Big)
    \\
    &\quad \geq
    -\|\phi\|_{L^\infty(0,T)}\Big(
    c(\underline\Phi^N)(\E(\ol{U}^N)+1) +\big|\mathcal K(\underline\Phi^N)-\tilde{\mathcal{K}}(\underline\Phi^N)\big|\mathcal{E}(\ol{U}^N)
    \Big)
\end{aligned}
\end{equation}
by~\ref{ass:K.conv}, where 
$c(\underline\Phi^N)$, $\mathcal K(\underline\Phi^N)$ and $\tilde{\mathcal{K}}(\underline\Phi^N)$
can be estimated by a constant only depending 
on an upper bound of $\|\Phi\|_{L^\infty(0,T;\Y)}$.
Therefore, we can apply Fatou's lemma to estimate 
\begin{align*}
    \liminf_{N \rightarrow \infty }&\int_0^T \underline{\phi}^N\left(\Psi(t,\ol{U}^N) -\langle A(t,\ol{U}^N), \underline{\Phi}^N\rangle +  \mathcal{K}(\underline{\Phi}^N) \E(\ol{U}^N) \right) \dd{t}  \\
    &\geq \int_0^T \liminf_{N \rightarrow \infty }  \underline{\phi}^N\left(\Psi(t,\ol{U}^N)-\langle A(t,\ol{U}^N), \underline{\Phi}^N\rangle +  \mathcal{K}(\underline{\Phi}^N) \E(\ol{U}^N)\right) \dd{t}.
\end{align*}
From the weak* convergence~\eqref{eq:conv.U.strong} pointwise in $[0,T]$
and the weak* lower semicontinuity assumed in~\ref{ass:K.lsc},
we further conclude 
\begin{align*}
     \liminf_{N \rightarrow \infty }  \underline{\phi}^N\left(\Psi(t,\ol{U}^N) -\langle A(t,\ol{U}^N), \underline{\Phi}^N\rangle +  \mathcal{K}(\Phi^N)\E(\ol{U}^N) \right) 
     \geq \phi\big(\Psi(t,U) -\langle A(t,U), \Phi\rangle + \KK\mathcal{E}(U) \big)
\end{align*}
in $[0,T]$, where we also use~\eqref{eq:conv.testfct}.
Combining the all these convergence results, we finally conclude 
from~\eqref{eq:envar.approximate} 
that the limiting pair $(U,E)$ satisfies the inequality
\[
\begin{aligned}
\int_0^T - \partial_t \phi \Big[E- \langle  U, \Phi\rangle \Big] + \phi \Big[ \langle U, \partial_t \Phi \rangle + \phi \big(\Psi(t,U) -\langle A(t,U), \Phi\rangle 
&+\KK \big[\E(U)-E\big] \Big] \dd{t} \\
&\leq\phi(0)\big(\mathcal{E}(U_0)-\langle U_0, \Phi(0)\rangle\big).
\end{aligned}
\]
Since $\phi \in \mathcal{C}^1_c([0,T))$ such that $\phi \geq 0$,
is arbitrary, 
this inequality is equivalent to 
the energy-variational inequality~\eqref{envar def1}
for a.a.~$s<t\in[0,T]$, including $s=0$ with 
$E(0+)- \langle  U(0+), \Phi(0)\rangle
=\mathcal{E}(U_0)-\langle U_0, \Phi(0)\rangle$,
see \cite[Lemma 2.5]{envarhyp}.
In particular, it holds $E(0+)=\E(U_0)$.
Since $\Phi \in \mathcal{C}^1([0,T];\Y\cap\dom\partial\E^*)$ is arbitrary,
we have shown that the limit pair $(U,E)$ is an energy-variational solution in the sense of Definition~\ref{sol def} with $E(0+)=\E(U_0)$.
\end{proof}

\subsection{Selection criteria}
\label{subsec:selection}

While the existence of energy-variational solutions is ensured by Theorem~\ref{main theorem1},
the solution set can multi-valued in general.
This motivates to introduce a additional admissibility criteria
to select solutions with certain properties.
In what follows, we assume that the assumptions \ref{space as}-\ref{ass:K.conv} are fulfilled, so that the solution set is nonempty
due to the existence result from Theorem~\ref{main theorem1}.

We start with the following criterion that selects an energy-variational solution $(U,E)$ in such a way that the energy defect $E-\E(U)$ vanishes at countably many 
selected points in time. 
Note that the occurring point evaluations make sense, compare Proposition~\ref{prop p v}.
A similar criterion with finitely many points was introduced in~\cite[Prop.~6]{envar}.

\begin{theorem}\label{thm points}
    Let $\{t_i\}_{i \in \mathbbm{N}} \subset [0,T]$ be a sequence of time points such that $t_i < t_{i+1}$ for all $  i \in \mathbbm{N} $. Then there exists an energy-variational solution $(U,E)$ with $\E(U(t_i))=E(t_i)$ for all $i \in \mathbbm{N} $. 
\end{theorem}
\begin{proof}
     We construct the solution inductively. Without loss of generality, let $t_0=0$. Let $i \in \{0,1,\ldots \}$. For $t_0=0$ the existence of such a solution follows from Theorem \ref{main theorem1}, because $E(0+)=\E(U_0)$. Assume that there exists an energy-variational solution $(U^n,E^n)$ in $[0,t_{i+1}]$ such that $\E(U^n(t_i))=E^n(t_i)$ for all $i \leq n$. Next, consider $U^n(t_{i+1})$ with $t_{i+1}>t_i$. Theorem \ref{main theorem1} provides the existence of a solution $(\bar{U},\bar{E})$, such that $\bar{U}(t_{i+1})=U^n(t_{i+1})$ and $\E(\bar{U}(t_{i+1}))=E^n(t_{i+1})$. Then the concatenation 
      \begin{equation*}
    (U^{n+1},E^{n+1}):=\begin{cases}
    (U^n,E^n) \quad &\text{if} \quad t <t_{i+1}, \\
        (\bar{U},\bar{E}) \quad &\text{if} \quad t \geq t_{i+1}, 
    \end{cases}
\end{equation*}
is again an energy-variational solution due to Proposition \ref{prop p iv}  and fulfills $\E(U^{n+1}(t_i))=E^{n+1}(t_i)$ for all $i\leq n+1$.
\end{proof}

Although the previous selection criterion yields a solution with vanishing energy defect 
at specified time points, this defect could still become large between such points.
The following result shows that one can also select an energy-variational solution
in such a way that the energy defect is smaller than any prescribed constant.
In particular, it shows that there are energy-variational solutions
that are close to weak solutions.
An analogous property was shown for dissipative weak solutions the Euler system in~\cite[Theorem 3.1]{selectEuler}. 

\begin{theorem}\label{thm eps sleeve}
    For all $\epsilon >0$ there exists an energy-variational solution $(U,E)$ such that 
    \[ 
    \E(U(t)) \leq E(t) \leq \E(U(t))+ \epsilon.\]
\end{theorem}

\begin{proof}
Let $(U,E)$ be an energy-variational solution satisfying Theorem \ref{main theorem1}. Then, let $t_\delta(U,E):=\sup \{t>0 \; |\;  E(s)\leq \E(U(s))+\delta \; \forall  s \leq t\}$. To see that $t_\delta(U,E)>0$, let us suppose the opposite, that $t_\delta(U)=0$. Then, for all $t>0$, there exists $s\in(0,T)$ such that $E(s) > \E(U(s))+\delta$. Let $s \searrow 0_+$. Then $E(s) - \E(U(s))>\delta>0$ implies that $\lim_{s \to 0_+}E(s) - \liminf_{s \to 0_+}\E(U(s)) \geq \delta >0$. However,  $E \in \mathrm{BV}([0,T])$, and $\liminf_{s \to 0_+}\E(U(s))=\E(U_0)$ by the same argument as in Proposition \ref{prop point ii}. Therefore we get $E(0_+)-\E(U_0) >0$, which contradicts the initial condition.  This shows that $t_\delta(U,E)>0$. Next, let us define an ordering such that $(U,E) \ll(V,F)$ if and only if $t_\delta(U,E)\leq t_\delta(V,F)$ and $(U,E)=(V,F)$ on $[0, t_\delta(U,E)]$. The solution set $\mathcal{S}(U_0)$ endowed with the relation $\ll$ is a partially ordered set. Let 
\[
(U^1,E^1) \ll (U^2,E^2) \ll \ldots \ll (U^n,E^n), \quad (U^i,E^i) \in \mathcal{S}(U_0) \quad \forall i=1,2,\ldots
\]
be an ordered chain. By Proposition~\ref{prop p i}, there exists a subsequence $(U^n, E^n)$  such that $U^n \rightharpoonup \tilde{U}$ in $L^\infty_{w^*}(0,T;\V^*) \cap \mathrm{BV}([0,T];\Y^*)$ and $E^n \stackrel{*}{\rightharpoonup}\tilde{E}$ in $\mathrm{BV}([0,T])$. It holds that $(\tilde{U}(t),\tilde{E}(t))=(U^n(t), E^n(t)) $ for all $ t < t_\delta(U^n,E^n)$, and for all $n=1,2,\ldots$. Now, either $t_\delta(\tilde{U},\tilde{E})=T$, meaning that $(\tilde{U},\tilde{E})$ is an upper bound for the chain, or $t_\delta(\tilde{U},\tilde{E})<T$ and we construct a concatenation 
\[
(U, E)(t, \cdot) =
\begin{cases}
(\tilde{U}, \tilde{E})(t, \cdot), & \text{for } 0 \leq t \leq t_\delta(\tilde{U},\tilde{E}), \\
(\hat{U}, \hat{E})(t - \tilde{T}, \cdot), & \text{for } t > t_\delta(\tilde{U},\tilde{E}).
\end{cases}
\]
where $(\hat{U}, \hat{E})$ is an energy-variational solution starting at time $t_\delta(\tilde{U},\tilde{E})$, with initial condition $U(t_\delta(\tilde{U},\tilde{E}))$, $ E(t_\delta(\tilde{U},\tilde{E}))$. Such constructed $(U,E)$ is an upper bound for the chain. By Zorn's lemma we conclude that there exists a solution in $\mathcal{S}(U_0)$ which is maximal with respect to relation $\ll$.
Either this solution, let us call it $(U,E)$, satisfies $t_\delta(U,E)=T$, or $t_\delta(U,E)<T$. The first case ends the proof, but in the latter case we can build another solution $(\tilde{U},\tilde{E})$, by concatenation as we did just above. The new solution would satisfy $(U,E) \ll (\tilde{U},\tilde{E})$, therefore contradicting the claim that $(U,E)$ with $t_\delta(U,E) <T$ was maximal.  
\end{proof}

While the previous two results focused on properties of the energy defect as a criterion for selection,
one may also select solutions by minimizing or maximizing 
certain functionals. 
Here we make use of the structural properties of the solution set,
established in Proposition~\ref{prop p iii} and Proposition~\ref{prop p i}.

\begin{theorem}
    \label{thm:selection.functional}
    Let $U_0\in\dom\E$, set $X:=L^\infty_{w^*}(0,T;\mathbbm{V}^*) \cap \mathrm{BV}([0,T]; \Y^*)\times\BV([0,T])$, and let
    $\mathcal I\colon X\to(-\infty,\infty]$ 
    be a weakly* lower semicontinuous functional.
    Then there exists an energy-variational solution $(U,E)$ that minimizes the functional among all energy-variational solutions with initial value $U_0$ and with $E(0+)=\E(U_0)$,
    that is,
    \[
    \mathcal I(U,E)=\min\big\{
    \mathcal I(V,F)\mid (V,F)\in \mathcal S(U_0,\E(U_0))
    \big\}.
    \]
    If $\mathcal K=\tilde{\mathcal K}$ and if $\mathcal I$ is strictly convex, then this minimizer is unique.
\end{theorem}

\begin{proof}
    The set $\mathcal S(U_0,\E(U_0))$ is non-empty by Theorem~\ref{main theorem1},
    and it is sequentially weakly* compact by Proposition~\ref{prop p i}.
    Therefore, any infimizing sequence has a limit in $S(U_0,\E(U_0))$,
    and by weak* lower semicontinuity of $\mathcal I$, 
    this limit is the asserted minimizer.
    Moreover, if $\mathcal K=\tilde{\mathcal K}$, 
    then $S(U_0,\E(U_0))$ is convex by Proposition~\ref{prop p iii},
    and the strict convexity of $\mathcal I$ implies the uniqueness of
    the minimizer.
\end{proof}

The standard example would be given by an integral functional
\[
\mathcal I(U,E)=\int_0^T j(t,U(t),E(t))\dd{t}
\]
for a suitable function $j\colon[0,T]\times\V^*\times[0,\infty)$
such that $j(t,\cdot,\cdot)$ is weakly* lower semicontinuous.
A simple choice would be to take $j(t,U,E)=E$ or $j(t,U,E)=\E(U)$, 
which would lead to the minimization of the auxiliary energy variable $E$
or the mechanical energy $\E(U)$
in the $L^1(0,T)$-norm.
By minimizing the energy, we maximize the dissipation,
so that these solutions might be considered maximally dissipative in a certain sense.
Of course, combinations of these properties, like $j(t,U,E)=\E(U)+E$, are possible.

A related choice that takes a different viewpoint would be the consideration 
of $j(t,U,E)=\E(U)-E$, the (negative) energy defect.
As seen in the existence proof in Subsection~\ref{subsec:existence},
the difference between $E$ and $\E(U)$ is due to the lack of strong convergence. 
Therefore, the energy defect quantifies to some extent the 
the oscillations and concentrations that make the difference between strong and weak convergence,
which may be regarded as a measure for turbulence in our fluid dynamical applications.
Minimizing the integral functional over $j(t,U,E)=\E(U)-E$ 
thus yields an energy-variational solution with maximal turbulence,
which adapts the concept from~\cite{EmilMaiximalTurb},
where maximally turbulent measure-valued solutions to the Euler equations were introduced.

Since the functional $\mathcal I$ takes into account global information in $[0,T]$,
minimal energy-variational solutions with respect to $\mathcal I$
need no longer satisfy the semi-flow property established in Proposition~\ref{prop p iv}.
However, in~\cite{breitfeireislhofmanova2020semiflowcompleteEuler,breitfeireslhofmanova2020semiflowisentropicEuler},
a semi-flow selection was established for dissipative weak solutions to the Euler equations
by successively minimizing functionals of the form
\[
\mathcal I(U,E)=\int_0^\infty e^{-\lambda t} F(U(t),E(t))\dd{t}
\]
for different $\lambda\in\R$ and a certain function $F$. 
Therefore, it might be possible to transfer this approach to the current setting
of energy-variational solutions.

\section{The Euler--Korteweg system}
\label{sec:eulerkorteweg}

In this section we provide an example for a system for which existence of energy-variational solutions can be shown by applying Theorem \ref{main theorem1},
and which is not treatable with the existence theory developed before in~\cite{envarhyp,viscoenvar},
namely the Euler--Korteweg system.

Let $\Omega \subset \R^d$, $d\in\N$, be a bounded Lipschitz domain and $T>0$.
We consider the Euler--Korteweg system
with a constant capillary coefficient and an adiabatic pressure law,
which reads
\begin{subequations}\label{au}
    \begin{align}
    \partial_t \rho + \diiv m &=0 \quad &&\text{in}\; \Omega \times (0,T) ,
    \label{eq:eulerkorteweg.cont}
    \\
   \partial_t m + \diiv\left( \frac{m \otimes  m}{\rho}\right) + \nabla( \rho^\gamma)&=\rho \nabla(\Delta \rho ) \quad &&\text{in} \;\Omega \times (0,T),  
   \label{eq:eulerkorteweg.mom}
   \\
   (\rho(0,\cdot),m(0,\cdot))&=(\rho_0,m_0) \quad &&\text{in} \;\Omega  .
\end{align}
Here, $\rho\colon(0,T)\times\Omega\to[0,\infty)$ and $m\colon(0,T)\times\Omega\to\R^d$
denote the density and the momentum of the fluid.
For simplicity, we have set all physical constants equal to $1$.
The system is equipped with impermeability boundary conditions 
\begin{align}\label{eq:eulerkorteweg.bdry}
  m \cdot n&=0{} = \nabla\rho\cdot n \quad \text{on} \quad (0,T) \times \partial \Omega  .
\end{align}\end{subequations}
Note that the momentum equation~\eqref{eq:eulerkorteweg.mom} can also be written as
\begin{equation}
\label{eq:eulerkorteweg.mom.K}
\partial_t m + \diiv\left( \frac{m \otimes  m}{\rho}\right) + \nabla( \rho^\gamma)=\diiv \mathbbm K
\quad \text{ in } \Omega \times (0,T)  
\end{equation}
with the Korteweg stress tensor 
$\mathbbm{K}= \Big( \diiv(\rho \nabla \rho) -\frac{1}{2}|\nabla \rho|^2\Big)\mathbbm{I}_d-\nabla \rho\otimes\nabla \rho$.
 
\subsection{Energy-variational framework}

To introduce a setting of energy-variational solutions to the Euler--Korteweg system, 
we define the total energy
\begin{equation}\label{eq:energy.ek}
\mathcal{E}(\rho,m)=\int_\Omega \eta(\rho(x),m(x))+\frac{|\nabla \rho(x)|^2}{2}\dd{x},
\end{equation}
where 
$\eta: \mathbbm{R} \times \mathbbm{R}^d \rightarrow [0,\infty]$ is given by 
\[
    \eta(\rho, m) = \begin{cases}
    \frac{|m|^2}{2\rho}+\frac{\rho^\gamma}{\gamma-1} \quad &\text{if} \quad \rho>0, \\
    0 \quad &\text{if}\quad (\rho,m)=(0,0),\\
    \infty \quad &\text{else}.
       \end{cases}
\]
We further define 
\begin{equation}\label{eq:space.testfct.mom}
    \mathcal{V}=\{\varphi \in \mathcal{C}^2(\overline{\Omega};\mathbbm{R}^d): \varphi\cdot n=0\text{ on } \partial\Omega\}.
\end{equation}
Then the regularity weight $\mathcal K\colon\mathcal V\to[0,\infty)$, which appears in the energy-variational formulation,
is chosen as 
\begin{equation}
\label{eq:K.lsc.ek}
\mathcal{K}(\varphi)
= 2\|(\nabla \varphi)_{\mathrm{sym},-}\|_{L^\infty(\Omega)}+\max\{(\gamma-1,1)\}\|(\diiv\varphi)_{-}\|_{L^\infty(\Omega)}.
\end{equation}

With these notation, 
we can introduce energy-variational solutions to~\eqref{au} by adapting 
Definition~\ref{sol def} as follows.

\begin{definition}\label{sol eu kor}
Let $(\rho_0,m_0)\in L^1(\Omega)\times L^1(\Omega;\R^3)$ such that $\E(\rho_0,m_0)<\infty$.
A triple $(\rho,m,E)\in L^\infty(0,T;H^1(\Omega) \cap L^\gamma(\Omega)) \times L^\infty(0,T;L^\frac{2\gamma}{\gamma -1}(\Omega;\R^d)) \times \mathrm{BV}([0,T])$ is called an energy-variational solution to the Euler--Korteweg system~\eqref{au} if $\mathcal{E}(\rho(t),m(t)) \leq E(t) $ for a.e.~$t \in (0,T)$ and  if
\begin{equation}
\label{eq:envar.ek}
\begin{aligned}
    \left[E- \int_\Omega \rho \, \psi + m \cdot\,\varphi \dd{x} \right]\bigg|_s^t &+\int_s^t \int_\Omega \rho \, \partial_t \psi+ m \cdot\nabla \psi + m \cdot \,\partial_t \varphi + \left(\frac{m \otimes m}{\rho}+ \rho^\gamma \mathbbm{I}_d\right):\nabla \varphi \dd{x} \dd{\tau} \\
      &+\int_s^t\int_\Omega  \rho\nabla \rho \cdot \nabla (\diiv
       \varphi) + 
 \frac{1}{2} |\nabla \rho|^2\diiv \varphi  + \nabla \rho \otimes \nabla \rho:\nabla \varphi \dd{x}\dd{\tau}\\
     &+ \int_s^t \mathcal{K}(\varphi)(\mathcal{E}(\rho,m)-E) \dd{\tau}\leq 0
\end{aligned}
\end{equation}
holds for all test functions $\psi\in\C^1([0,T];\C^1(\Omega))$
and $\varphi \in \mathcal{C}^1([0,T];\mathcal V) $,
and for a.e.~$ s<t$, including $s=0$ with $\rho(0)=\rho_0,\, m(0)=m_0$.
\end{definition}

One readily sees that~\eqref{eq:envar.ek} is derived 
by adding the energy inequality $E\big|^t_s\leq 0$ 
and the term with the regularity weight 
to a weak formulation of~\eqref{au},
which can be directly derived using the representation~\eqref{eq:eulerkorteweg.mom.K} of the momentum equation and integration by parts,
where we include the boundary conditions~\eqref{eq:eulerkorteweg.bdry} in a weak sense.

We apply the previously developed theory to show 
the existence of energy-variational solutions to the Euler--Korteweg system.
    
\begin{theorem}\label{thm:existence.EulerKorteweg}
Let $\Omega \subset \mathbbm{R}^d$ be a bounded domain with $\C^{1,1}$-boundary. 
Then for all initial data $(\rho_0,m_0) \in L^1(\Omega)\times L^1(\Omega;\R^3)$ such that $\mathcal{E}(\rho_0,m_0)< \infty$,  the system \eqref{au} admits an energy-variational solution as in Definition \ref{sol eu kor}, with $\E(\rho_0,m_0)=E(0+)$.
\end{theorem}

In order to obtain solutions by using the abstract framework introduced Section~\ref{sec:abstract},
we encode the mean value of the density $\rho$ in the functional framework.
Observe that this mean value is invariant during the evolution.
Indeed, setting $s=0$, $\phi=0$ and $\psi\equiv\alpha$ for some constant $\alpha\in\R$, 
we obtain
\[
E(t)-\alpha\int_\Omega\rho(t,x)\dd{x}\leq E(0)-\alpha\int_\Omega\rho_0(x)\dd{x}
\]
for a.a.~$t\in[0,T]$.
As $\alpha\in\R$ is arbitrary, 
this implies that $\int_\Omega\rho\dd{x}$ is constant in time.

The constraint to mean-free functions will be used for finding a  regularity weight $\tilde{\mathcal K}$
suitable to satisfy assumption~\ref{ass:K.conv}.
In order to keep the linear structure of the spaces,
for a given initial value $\rho_0$, we set 
\[
\overline\rho
:=\frac{1}{|\Omega|}\int_\Omega\rho_0(x)\dd{x},
\qquad
h(t,x):=\rho(t,x)-\overline\rho,
\]
and we show existence of solutions in the form $(\rho,m)=(h+\overline\rho,m)$,
where $h$ has vanishing mean value.

We thus consider state variables $(h,m)$ and test functions $(\psi,\varphi)$ in the spaces
\begin{equation}\label{spaces}
\begin{split}
 \mathbbm{V}^* =\left(H_{(0)}^1(\Omega) \cap L^\gamma(\Omega) \right)\times L^\frac{2\gamma}{\gamma -1}(\Omega;\R^d), \qquad
    \mathbbm{Y}=\mathcal{C}^1_{(0)}(\overline{\Omega})\times \mathcal{V}, 
    \end{split}
\end{equation} 
with $\mathcal V$ from~\eqref{eq:space.testfct.mom}.
Here $H_{(0)}^1(\Omega)$ and $\mathcal{C}^1_{(0)}(\overline{\Omega})$ 
denote the subspaces of $H^1(\Omega)$ and $\C^1(\overline{\Omega})$ 
of elements with vanishing mean.
Observe that the state space $\V^*$ is the dual space of the reflexive and separable space $\V=\big((H_{(0)}^{1}(\Omega))^* + L^{\frac{\gamma}{\gamma-1}}(\Omega)\big)\times L^{\frac{2\gamma}{\gamma+1}}(\Omega;\R^d)$. 

The energy functional $\ol{\E}\colon\V^*\to[0,\infty]$ is given by
\begin{equation}
\label{eq:energy.ek.hm}
\ol\E(h,m):=\E(h+\ol\rho,m)=\int_\Omega \eta(h(x)+\ol\rho,m(x))+\frac{|\nabla h(x)|^2}{2}\dd{x}
\end{equation}
with $\E$ from~\eqref{eq:energy.ek}.
In this system the dissipation potential vanishes, that is, $\Psi \equiv 0$,
and the 
operator $A\colon\dom \ol\E\to\Y^*$ is given by
\begin{equation}
\label{A}
\begin{aligned}
\langle A(h,m),(\psi, \varphi)\rangle
&=\int_\Omega -m \cdot \nabla \psi -\left( \frac{m \otimes m}{h+\ol\rho}+(h+\ol\rho)^\gamma \mathbbm{I}_d\right):\nabla \varphi  \dd{x}   \\
 &\quad+\int_\Omega-(h+\ol\rho) \nabla h \cdot\nabla( \diiv \varphi)-\frac{1}{2}|\nabla h|^2 \; \diiv \varphi - \nabla h \otimes \nabla h :\nabla \varphi \dd{x}.
\end{aligned}
\end{equation}
Additional to the regularity weight~\eqref{eq:K.lsc.ek},
which yields sufficient lower semicontinuity properties in the sense of assumption~\ref{ass:K.lsc}, 
we need a larger regularity weight 
$\tilde{\mathcal K}:\mathcal V\rightarrow \mathbbm{R}$
given by
\begin{equation}
    \label{eq:K.conv.ek}
\tilde{\mathcal{K}}(\varphi)
=2 \|(\nabla \varphi)_{\mathrm{sym},-}\|_{L^\infty(\Omega)}+\max\{(\gamma-1,3)\}\|(\diiv\varphi)_{-}\|_{L^\infty(\Omega)} + 2c_P\|\nabla\diiv\varphi\|_{L^\infty(\Omega)},
\end{equation}
in order to satisfy the convexity assumption~\ref{ass:K.conv}.
Here, $c_P>0$ denotes the smallest constant such that the Poincar\'e inequality
\begin{equation}
    \label{eq:poincare}
    \forall h\in H^1_{(0)}(\Omega):\quad
    \int_\Omega|h|^2\dd{x}\leq c_P^2\int_\Omega|\nabla h|^2\dd{x}
\end{equation}
is satisfied.

\subsection{Existence of energy-variational solutions}

To prove Theorem~\ref{thm:existence.EulerKorteweg}, we show that 
the assumptions~\ref{space as}--\ref{ass:K.conv} are satisfied,
so that we can apply the general result from Theorem~\ref{main theorem1}.
We first collect properties of the energy functional.

\begin{lemma}[Properties of $\E$]\label{E coerc}
For $\ol\rho\geq 0$, let the functional $\ol\E\colon\V^*\to[0,\infty]$ be defined by~\eqref{eq:energy.ek.hm}.
is convex and weakly* lower semicontinuous
and there exists $c>0$ such that
\begin{equation}
    \label{eq:Egrowth.ek}
    \E(\rho,m):=\ol\E(\rho-\ol\rho,m) \geq c \Big(\|m\|_{L^{\frac{2\gamma}{\gamma+1}}(\Omega)}^{\frac{2\gamma}{\gamma+1}} + \|\rho\|^{\gamma}_{L^\gamma(\Omega)}+\|\nabla \rho\|_{L^2(\Omega)}^2\Big) 
\end{equation}
for all $(\rho,m)$ such that $(h,m)=(\rho-\ol\rho,m)\in\V^*$.
In particular, $\ol\E$ is coercive in the sense of~\eqref{eq:coercive}, has superlinear growth in the sense of~\eqref{eq:superlinear}
and satisfies
$\dom \partial\ol\E^*=\dom \ol\E^*=\V^{**}=\V$.
\end{lemma}
\begin{proof}
    The functionals $\E$ and $\ol\E$ are both convex and weakly lower semicontinuous 
    as integral functionals with nonnegative convex integrand. 
    Since $\V$ is reflexive, this also implies weak* lower semicontinuity.
    To show~\eqref{eq:Egrowth.ek}, let $q=\frac{\gamma+1}{\gamma}$ with $q'=\gamma+1$.
    If $\rho=h+\ol\rho>0$, then Young's inequality implies 
\begin{align*}
     \E(\rho,m) &- \frac{\|\rho\|^{\gamma}_{L^\gamma(\Omega)}}{2(\gamma-1)}-\frac{1}{2}\|\nabla \rho\|_{L^2(\Omega)}^2 
     = \int_\Omega \frac{|m|^2}{2\rho}+\frac{\rho^\gamma}{2(\gamma-1)} \dd{x} = \int_\Omega \Big(\frac{|m|^2}{2\rho}\Big)^{\frac{\gamma}{\gamma+1}q}+\Big(\frac{\rho}{(2(\gamma-1))^{\frac{1}{\gamma}}}\Big)^{\frac{\gamma}{\gamma+1}q'} \dd{x} \\
    &\geq c(\gamma)\int_\Omega \Big(\frac{|m|^2}{2\rho}\,\frac{\rho}{(2(\gamma-1))^{\frac{1}{\gamma}}}\Big)^{\frac{\gamma}{\gamma+1}}\dd{x} =c(\gamma) \int_\Omega  |m|^{\frac{2\gamma}{\gamma+1}}\dd{x}=c(\gamma)\|m\|_{L^{\frac{2\gamma}{\gamma+1}}(\Omega)}^{\frac{2\gamma}{\gamma+1}}
 \end{align*}
 for some constant $c(\gamma)>0$.
This shows~\eqref{eq:Egrowth.ek} for $\rho>0$. 
The other cases can be treated similarly or are trivial.
In particular, $\E$ satisfies~\eqref{eq:coercive} and~\eqref{eq:superlinear} and, in accordance with Proposition \ref{prop:Estar}, 
we conclude $\dom \ol\E^*=\V^{**}=\V$ by reflexivity of $\V$.
\end{proof}

In the next lemma, we consider weak solutions to the problem 
\[
\begin{aligned}
    \Delta \rho + \frac{\gamma}{\gamma-1}|\rho|^{\gamma - 2}\rho+\lambda&=g 
    && \text{in }\Omega,
    \\
    n \cdot \nabla \rho &=0 
    &&\text{on }\partial \Omega,
\end{aligned}
\]
with the constraint $\lambda\in\partial\mathcal G(\rho)$, where 
\begin{equation}
    \label{eq:indicator.fctl}
    \mathcal G\colon H^1(\Omega)\cap L^\gamma(\Omega)\to[0,\infty],
    \qquad
    \mathcal G(\rho)=\int_\Omega\iota_{[0,\infty)}\dd{x}.
\end{equation}
This is a step towards identifying elements of the subdifferential of the energy $\ol\E$, 
which is necessary for checking the validity of the assumption \ref{ass energy bal}.

\begin{lemma}\label{reg rho}
Let $\Omega\subset\R^d$ be a bounded domain with $\C^{1,1}$-boundary,
and let
$g \in C(\overline{\Omega})$. Then there exists a unique $\rho \in H^1(\Omega)\cap L^\gamma(\Omega), \rho \geq 0$ such that  
    \begin{equation}
        \label{eq:reg.rho.pde}
        \int_\Omega 
        \nabla \rho\cdot \nabla \varphi +\Big( \frac{\gamma}{\gamma-1}\rho^{\gamma - 1}+\lambda \Big) \varphi \dd{x} 
        =\int_\Omega g \varphi \dd{x} 
       \end{equation}
    for all $\varphi \in H^1(\Omega)\cap L^\gamma(\Omega)$,
    where $\lambda\in\partial\mathcal G(\rho)$ with $\mathcal G$ from~\eqref{eq:indicator.fctl}
    Moreover, 
    $\lambda\in L^p(\Omega)$ and $\rho\in W^{2,p}$ for any  $1\leq p <\infty$,
    and $\lambda(x) \in \partial \iota_{[0,\infty)}(\rho(x))$ for a.a.~$x\in\Omega$.
\end{lemma}

\begin{proof}
The weak formulation~\eqref{eq:reg.rho.pde} can be regarded as the Euler--Lagrange equations
for the functional $\mathcal D=\mathcal F+\mathcal G$ with
\[
\mathcal F(\rho)=\int_\Omega|\nabla\rho|^2+\frac{1}{\gamma-1}\rho^\gamma-g\rho\dd{x}.
\]
Since $\mathcal D$ is strictly convex, proper and lower semicontinuous, 
it has a unique minimizer $\rho\in H^1(\Omega)\cap L^\gamma(\Omega)$,
which yields the existence of a unique weak solution to~\eqref{eq:reg.rho.pde}.

To show the increased integrability,
we use the Yosida approximation for $\iota_{[0,\infty)}(\rho)$, defined as 
\begin{align*}
    f_\epsilon(\rho)=\inf_{\eta \in \R }\{\frac{1}{2\epsilon}|\eta-\rho|^2+\iota_{[0,\infty)}(\rho)\}=\inf_{\eta \in \R,\eta \geq 0 }\{\frac{1}{2\epsilon}|\eta-\rho|^2\}=\begin{cases}
       0\quad &\text{if } \rho \geq 0, \\
       \frac{1}{2\epsilon}|\rho|^2 \quad  &\text{if } \rho < 0.
    \end{cases}
\end{align*}
We first consider a regularized version of the problem, given by the weak formulation
\begin{align}\label{reg pde}
  \int_\Omega \nabla \rho_\epsilon\cdot \nabla \varphi +\Big( \frac{\gamma}{\gamma-1}|\rho_\epsilon|^{\gamma - 2}\rho_\epsilon+f'_\epsilon(\rho_\epsilon) \Big) \varphi \dd{x} &=\int_\Omega g \varphi \dd{x}
\end{align}
for all $\varphi\in H^1(\Omega)\cap L^{\gamma}(\Omega)$.
 From \cite[Theorem 4.7]{troltzsch}  it follows that there exists $\rho_\epsilon \in H^1(\Omega)\cap L^\infty(\Omega)$ that satisfies \eqref{reg pde}.
To obtain \textit{a priori} estimates on $f'_\epsilon(\rho_\epsilon)$, we  test the equation \eqref{reg pde} with $f'_\epsilon|f'_\epsilon|^{p-2}$. Since
\[
f'_\epsilon(\rho(x)) = 
\begin{cases}
0, & \text{if } \rho(x) \ge 0, \\[2mm]
\frac{\rho(x)}{\epsilon}, & \text{if } \rho(x) < 0,
\end{cases}
\]
we see that $|f'_\epsilon(\rho_\epsilon)|^{p-2}f'_\epsilon(\rho_\epsilon)$ is a viable test function. After a short calculation, we obtain 
\[
\begin{split}
\int_\Omega 
    \Big[(p-1)|f'_\epsilon(\rho_\epsilon)|^{p-2}f^{''}_\epsilon(\rho_\epsilon) \Big] |\nabla \rho_\epsilon|^2 
    + \Big(\frac{\gamma}{\gamma+1}|\rho_\epsilon|^{\gamma-2}\rho_\epsilon 
    &+ f'_\epsilon(\rho_\epsilon)
\Big)
|f'_\epsilon(\rho_\epsilon)|^{p-2} f'_\epsilon(\rho_\epsilon) \dd{x}
\\
&=\int_\Omega 
g \, |f'_\epsilon(\rho_\epsilon)|^{p-2} f'_\epsilon(\rho_\epsilon)
\dd{x}.
\end{split}
\]
Due to $f''_\epsilon(\rho_\epsilon)\geq 0$ 
and $f'_\epsilon(\rho_\epsilon)\rho_\epsilon\geq 0$,
all terms on the left-hand side are nonnegative. 
Omitting the first and the second term, we obtain
    \begin{align*}
\int_\Omega |f'_\epsilon(\rho_\epsilon)|^{p} \dd{x}
\leq \int_\Omega 
g \, |f'_\epsilon(\rho_\epsilon)|^{p-2} f'_\epsilon(\rho_\epsilon) 
\dd{x} \leq \frac{1}{p}\|g\|_{L^p(\Omega)}^p+\frac{p-1}{p}\|f'_\epsilon(\rho_\epsilon)\|_{L^{p}(\Omega)}^{p}
    \end{align*}
    by H\"older's and Young's inequality,
which implies the uniform bound 
\begin{align}\label{a priofi f'}
    \|f'_\epsilon(\rho_\epsilon)\|_{L^{p}(\Omega)}^{p} &\leq \|g\|_{L^p(\Omega)}^p.
\end{align}
To derive uniform estimates of $\rho_\epsilon$ in $L^p(\Omega)$,
first consider $p\geq\gamma$. 
We test \eqref{reg pde} with $|\rho_\epsilon|^{p-\gamma}\rho_\epsilon$,
which is an element of $H^1(\Omega)\cap L^\infty(\Omega)$. 
This yields 
\[
    \int_\Omega  (p-\gamma+1)|\rho_\epsilon|^{p-\gamma}|\nabla \rho_\epsilon|^2
    +\frac{\gamma}{\gamma-1}|\rho_\epsilon|^{p}
    + |\rho_\epsilon|^{p-\gamma} f'_\epsilon(\rho_\epsilon)\rho_\epsilon \dd{x}
=
\int_\Omega 
g |\rho_\epsilon|^{p-\gamma}\rho_\epsilon \dd{x}
\leq \frac{1}{\delta p}\|g\|_{L^q}^q+ \frac{\delta}{p'}\|\rho_\epsilon\|_{L^p}^{p}.
\]
for $\delta>0$ 
and $q=\frac{p}{\gamma-1}$
by H\"older's and Young's inequalities.
Again, all terms on the left-hand side are nonnegative,
and we can omit the first and the third one.
Choosing $\delta>0$ sufficiently small reveals
\begin{equation}
\label{eq:apriori.rhoeps}
\| \rho_\varepsilon\|_{L^p(\Omega)}^p
\leq c\|g\|_{L^q(\Omega)}^q
\end{equation}
for some $c>0$.
As $\Omega$ is bounded, this gives a uniform bound for all $p\in[1,\infty)$.
To derive \textit{a priori} estimates for $\nabla \rho_\epsilon$, we 
take $p=\gamma$ in the previous identity, 
omit the second and third term and argue similarly
to obtain
\begin{equation}
\label{eq:apriori.nablarho}
\|\nabla \rho_\epsilon\|_{L^2}^2+\Big(\frac{\gamma}{\gamma - 1} -\frac{1}{\gamma}\Big)\|\rho_\epsilon\|_{L^\gamma}^\gamma\leq \frac{1}{\gamma'} \|g\|_{L^{\gamma'}}^{\gamma'} ,
\end{equation}
which also implies a uniform bound on $\|\rho_\epsilon\|_{H^1(\Omega)}$ 
by the Gagliardo--Nirenberg inequality.

By the uniform bounds~\eqref{a priofi f'}, \eqref{eq:apriori.rhoeps} and~\eqref{eq:apriori.nablarho},
we can pass to a subsequence such that $f'(\rho_\epsilon)\rightharpoonup\lambda$ in $L^p(\Omega)$ 
and $\rho_\varepsilon\rightharpoonup\rho$ in $L^p(\Omega)\cap H^1(\Omega)$
for some $\lambda\in L^p(\Omega)$ and $\rho\in L^p(\Omega)\cap H^1(\Omega)$.
The Rellich--Kondrachov theorem implies that $\rho_\varepsilon\to\rho$ in 
$L^2(\Omega)$,
which implies convergence a.e.~in $\Omega$ for a subsequence.
Moreover, by~\eqref{a priofi f'} we have
\[
\frac{1}{\epsilon^p}\int_\Omega|\rho(x)|^p \mathbbm{1}_{\{\rho_\epsilon<0\}}\dd{x}
=\|f'(\rho_\epsilon)\|_{L^p(\Omega)}\leq \|g\|_{L^p(\Omega)}^p,
\]
Letting $\epsilon \to 0$, we see that $ |\rho_\epsilon|^2 \, \mathbbm{1}_{\{\rho_\epsilon < 0\}} \to 0$ in $L^1(\Omega)$, which implies $\rho \geq 0$.
Using Vitali's convergence theorem for the nonlinear term,
we can now pass to the limit $\epsilon\to0$ in~\eqref{reg pde}
to obtain the weak formulation~\eqref{eq:reg.rho.pde}
for any $\varphi\in H^1(\Omega)\cap L^\gamma(\Omega)$.

It remains to show the pointwise characterization of $\lambda(x)\in\partial\iota_{[0,\infty)}(\rho(x))$. 
With the functional $\mathcal F$ from above, 
we can write~\eqref{reg pde} as $-D\mathcal F(\rho_\epsilon)=f'_\epsilon(\rho)$.
By the the convexity of $f_\epsilon$, we thus have
\[
\langle -D\mathcal{F}(\rho_\epsilon),\varphi-\rho_\epsilon \rangle  \leq \int_\Omega f_\epsilon(\varphi) - f_\epsilon(\rho_\epsilon) \dd{x}
\leq \int_\Omega f_\epsilon(\varphi) - f_1(\rho_\epsilon) \dd{x}
\]
for $\epsilon\leq 1$ as $f_1\leq f_\epsilon$.
Let us only consider $\varphi \geq 0$ for the moment. Then $f_\epsilon(\varphi)=0$. We further have $\langle- D\mathcal{F}(\rho_\epsilon),\varphi \rangle=\langle f_\epsilon'(\rho_\epsilon),\varphi\rangle\to\langle\lambda,\varphi\rangle$.
It remains to investigate the convergence of 
\[
    \langle D\mathcal{F}(\rho_\epsilon),\rho_\epsilon \rangle+\int_\Omega f_1(\rho_\epsilon) \dd{x}=\int_\Omega |\nabla \rho_\epsilon|^2 
+  \frac{\gamma}{\gamma -1} |\rho_\epsilon|^{\gamma} - g \rho_\epsilon  +f_1(\rho_\epsilon)\dd{x}.
\]
As a mapping of $\rho_\varepsilon$,
this defines a 
convex and lower semicontinuous functional, so that 
\begin{align*}
    \int_\Omega |\nabla \rho|^2 
+  \frac{\gamma}{\gamma -1} |\rho|^{\gamma} - g \rho +f_1(\rho) \dd{x}\leq \liminf_{\epsilon \to 0}\int_\Omega |\nabla \rho_\epsilon|^2 
+  \frac{\gamma}{\gamma -1} |\rho_\epsilon|^{\gamma} - g \rho_\epsilon  +f_1(\rho_\epsilon)\dd{x}.
\end{align*}
Since $\rho \geq 0$ a.e., we have
$f_1(\rho(x))=0=\iota_{[0,\infty)}(\rho(x))$ a.e.~in $\Omega$,
and we arrive at
\[
\langle \lambda, \varphi -\rho \rangle=\langle -D\mathcal{F}(\rho),\varphi-\rho \rangle  \leq \int_\Omega \iota_{[0,\infty)}(\varphi(x)) - \iota_{[0,\infty)}(\rho(x)) \dd{x}.
\]
Note that this inequality holds for all $\varphi \in H^1(\Omega)\cap L^\gamma(\Omega)$, since it is trivially satisfied if $\varphi\geq 0 $ does not hold. 
This variational inequality gives the pointwise identification $\lambda(x)\in \partial \iota_{[0,\infty)}(\rho(x))$.
Finally, due to $\C^{1,1}$-regularity of the boundary of $\Omega$, 
classical elliptic theory for the Neumann--Laplace problem,
see~\cite[Theorem 2.2.2.5]{grisvard1985elliptic} for example, implies $\rho\in W^{2,p}(\Omega)$.
This concludes the proof of the lemma.  
\end{proof}

With the help of Lemma~\ref{reg rho}, 
we can further study the conjugate functional $\ol\E^*$
and characterize the subdifferential $\partial\ol\E^*$.

\begin{lemma}\label{subdiff eu kort}
Let $\ol\rho\geq 0$ and 
let $\ol\E\colon\V^*\to[0,\infty]$ be defined by~\eqref{eq:energy.ek.hm}.
Then, for all $(\xi,\zeta) \in \mathcal{C}^1_{(0)}(\Omega) \times \mathcal{V}$ there exist $(h,m) \in \partial \ol\E^*(\xi,\zeta)$ and $\lambda \in L^p(\Omega)$ 
with $\lambda(x)\in\partial\iota_{[0,\infty)}(\rho(x))$ for a.a.~$x\in\Omega$,
where $\rho=h+\ol\rho$, that satisfy 
$m=\rho\zeta$ and
\begin{equation}
    \label{eq:subdiff.pde.ek}
    \begin{aligned}
          \int_\Omega \nabla \rho\cdot \nabla \chi +\Big( \frac{\gamma}{\gamma-1}\rho^{\gamma - 1}+\lambda \Big) \chi \dd{x} &=\int_\Omega \xiz \chi\dd{x}   
        \end{aligned}
\end{equation}
for all $\chi \in \V$. 
Moreover, $\rho,h\in W^{2,p}(\Omega)$ and $m\in W^{2,p}(\Omega;\R^d)$ for any $1\leq p<\infty$.
In particular, $\rho\geq 0$ and $\lambda\leq 0$ a.e.~in $\Omega$,
and $\lambda(x)=0$ if $\rho(x)>0$ for a.a.~$x\in\Omega$.
\end{lemma}

\begin{proof}
Let $(\xi,\zeta) \in \mathcal{C}^1_{(0)}(\Omega) \times \mathcal{V}=\dom\partial\ol\E$, 
compare Lemma~\ref{E coerc}.
 Since $\ol\E$ is a proper, convex, and lower semicontinuous functional, we have
\[
(h,m) \in \partial \ol\E^*(\xi,\zeta) \Leftrightarrow (\xi,\zeta) \in \partial \ol\E(h,m),
\]
by~\cite[Proposition 2.33]{barbu} for instance.
In particular, $(h,m)\in\dom\ol\E$, 
which implies $\rho:=h+\ol\rho\geq 0$.
In virtue the definition of $\ol\E$, see~\eqref{eq:energy.ek.hm}, 
this directly yields $m=(h+\ol\rho)\zeta=\rho\zeta$,
and $(\xi,\zeta) \in \partial \ol\E(h,m)$
reduces to $0\in\partial\mathcal C(h)$, 
where ${\C}:H^1_{(0)}(\Omega)\cap L^\gamma(\Omega) \rightarrow \R\cup\{\infty\}$
is given by $\C:=\mathcal F+\mathcal G$ with
\[
{\mathcal{F}}(h)=\int_\Omega \frac{|h+\ol\rho|^\gamma}{\gamma-1}+\frac{|\nabla h|^2}{2}-\rho\Big(\xi+\frac{|\zeta|^2}{2}\Big)\dd{x},
\]
and $\mathcal G$ as in~\eqref{eq:indicator.fctl}
Since $\mathcal{F}$ is Gateaux differentiable, 
the condition $0\in\partial\C(h)$ corresponds
to 
\begin{equation}
    \label{eq:subdiff.pde.ek.1}
    \begin{aligned}
          \int_\Omega \nabla h\cdot \nabla \chi +\frac{\gamma}{\gamma-1}(h+\ol\rho)^{\gamma - 1}\chi\dd{x} +\langle\lambda,\chi\rangle &=\int_\Omega \xiz \chi\dd{x}   
        \end{aligned}
\end{equation}
for all $\chi \in \V=H^1_{(0)}(\Omega)\cap L^\gamma(\Omega)$
and for some $\lambda\in\partial\mathcal G(h)\subset \V^*$.
Then $\lambda$ can be extended to a continuous functional on $H^1(\Omega)\cap L^\gamma(\Omega)$
with $\langle\lambda,1\rangle=0$,
and~\eqref{eq:subdiff.pde.ek.1} for all $\chi\in\V$ is equivalent to
\[
    \begin{aligned}
          \int_\Omega \nabla h\cdot \nabla \chi +\frac{\gamma}{\gamma-1}(h+\ol\rho)^{\gamma - 1}\chi\dd{x} +\langle\lambda,\chi\rangle 
          =\int_\Omega \Big(\xi+\frac{|\zeta|^2}{2}+\frac{1}{|\Omega|}\int_\Omega \frac{\gamma}{\gamma-1} (h+\ol\rho)^{\gamma - 1}-\frac{|\zeta|^2}{2}\dd{x}\Big)\chi\dd{x}
    \end{aligned}
\]
for all $\chi\in H^1(\Omega)\cap L^\gamma(\Omega)$.
The term in parenthesis defines a continuous function,
so that the regularity result from Lemma~\ref{reg pde} implies 
the asserted regularity of $h=\rho-\ol\rho$ and $\lambda$,
so that~\eqref{eq:subdiff.pde.ek.1} is equivalent to~\eqref{eq:subdiff.pde.ek}.
\end{proof}

We next show that the operator $A$, as defined in~\eqref{A}, satisfies~\ref{ass energy bal}.

\begin{lemma}\label{lem:energybalance.ek}
    For every $(\psi,\varphi)\in\Y$
    there exists $(\rho,m)\in\partial\ol\E^*(\psi,\varphi)$ such that
    \[\langle A(h,m),(\psi,\varphi)\rangle=0.\]
\end{lemma}

\begin{proof}
    Let $(\rho,m)\in\partial\ol\E^*(\psi,\varphi)$ and $\lambda$ be as in Lemma~\ref{subdiff eu kort}.
    Using $m=\rho\varphi$ with $\rho=h+\ol\rho$, we have
    \[
    \begin{aligned}
        \langle A(h,m),(\psi, \varphi)\rangle
        &=-\int_\Omega \rho\varphi \cdot \nabla \psi +  \rho\varphi\cdot\nabla\frac{|\varphi|^2}{2}+ \rho^\gamma \diiv\varphi \dd{x}
        \\
        &\qquad-\int_\Omega \nabla\rho\cdot\nabla(\rho\diiv\varphi) 
        -\frac{1}{2}|\nabla\rho|^2\diiv\varphi
        +\nabla\rho\otimes\nabla\rho:\nabla\varphi\dd{x}.
    \end{aligned}
    \]
    Since $\rho\in W^{2,p}(\Omega)$ for any $p\in(1,\infty)$, we have the identity
    \[
    -\int_\Omega\frac{1}{2}|\nabla\rho|^2\diiv\varphi\dd{x}
    =\int_\Omega\nabla\rho\cdot(\nabla^2\rho\,\varphi)\dd{x}
    =\int_\Omega\nabla\rho\cdot\nabla(\nabla\rho\cdot\varphi)\dd{x}
    -\int_\Omega\nabla\rho\cdot(\nabla\varphi\nabla\rho)\dd{x},
    \]
    which yields, together with integration by parts, that
    \[
    \begin{aligned}
        \langle A(h,m),(\psi, \varphi)\rangle
        &=\int_\Omega \diiv(\rho\varphi) \Big(\psi+\frac{|\varphi|^2}{2}\Big)- \rho^\gamma \diiv\varphi - \nabla\rho\cdot\nabla\diiv(\rho\varphi) 
        \dd{x}.
    \end{aligned}
    \]
    Since $\diiv(\rho\varphi)$ is mean-value free, we have $\diiv(\rho\varphi)\in\V$,
    and we can use the weak formulation~\eqref{eq:subdiff.pde.ek}
    to obtain
    \[
    \begin{aligned}
        \langle A(h,m),(\psi, \varphi)\rangle
        &=\int_\Omega \diiv(\rho\varphi) \Big(\frac{\gamma}{\gamma-1}\rho^{\gamma-1}+\lambda \Big) - \rho^\gamma \diiv\varphi
        \dd{x}
        \\
        &=\int_\Omega  \frac{1}{\gamma-1}\rho^{\gamma}\diiv\varphi+\rho \frac{\gamma}{\gamma-1}\rho^{\gamma-1}\nabla\rho\cdot\varphi+
        \lambda\diiv(\rho\varphi)
        \dd{x}
        =\int_\Omega \lambda\diiv(\rho\varphi)
        \dd{x}.
    \end{aligned}
    \]
    by the divergence theorem and $\varphi\cdot n=0$.
Since $\lambda \in \partial \iota_{[0,\infty)}(\rho)$ pointwise, 
on the set where $\rho(x)>0$, we have $\lambda(x)=0$. 
Since on the set where $(\rho\varphi)(x)=0$, the gradient $\nabla( \rho \varphi)=0$ vanishes a.e., by~\cite[Corollary 2.1]{chipot} for instance,
we conclude that  $\langle A(h,m),(\psi, \varphi)\rangle=0$.
\end{proof}

In the next lemmas we ensure the validity of the assumptions~\ref{ass:K.lsc} and~\ref{ass:K.conv}.
We first show that the choice of $\tilde{\mathcal K}$ as in~\eqref{eq:K.conv.ek} 
induces sufficient convexity.

\begin{lemma}\label{convex ek}
    Let $\ol\rho>0$. 
    For each $(\psi,\varphi) \in \Y$,
    the mapping 
    \[
    (h,m) \mapsto - \langle A(h,m),(\psi,\varphi) \rangle +\tilde{\mathcal{K}}(\varphi)\ol\E(h,m),
    \]
    with $A$ and $\tilde{\mathcal K}$ as in~\eqref{A} and~\eqref{eq:K.conv.ek},
    is convex and weakly* lower semicontinuous in $\V^*$.
    Moreover, if $\|(\psi,\varphi)\|_{\Y}\leq R$ for some $R>0$,
    then there exists a constant $c(R)>0$ such that
    \begin{equation}
        \label{eq:AKlowerbound.ek}
        -\langle A(h,m),(\psi,\varphi) \rangle +\tilde{\mathcal{K}}(\varphi)\ol\E(h,m)
        \geq -c\big(\ol\E(h,m)+1\big)
    \end{equation}
    for all $(h,m)\in\V^*$.
\end{lemma}

\begin{proof}
Clearly, we can change to the state variables $(\rho,m)=(h+\ol\rho,m)$,
consider the energy $\E$ as in~\eqref{eq:energy.ek},
and study properties of the mapping 
\begin{equation}
    \label{eq:fct.convexified.ek}
(\rho,m) \mapsto - \langle A(\rho-\ol\rho,m),(\psi,\varphi) \rangle +\tilde{\mathcal{K}}(\varphi)\E(\rho,m)
\end{equation}
for $(\rho-\ol\rho,m)\in\V^*$ instead.
Since this mapping
is an integral functional, and $\V^*$ is reflexive,
its convexity and weakly* lower semicontinuity follow 
if the integrand is pointwise convex in $\Omega$. 
We use this fact and we decompose the integrand into several contributions,
which we treat separately
and combine it with the associated term from $\tilde{\mathcal K}(\varphi)\E(h,m)$.

The first contribution in $-A$ from~\eqref{A}, 
namely the integral over $m \cdot \nabla \psi$, 
is linear in $m$ and thus convex.

For the second term in~\eqref{A}, we observe that
\begin{equation*}
    \frac{m \otimes m}{\rho}: \nabla \varphi + \frac{|m|^2}{2\rho}{2\|(\nabla \varphi)_{\sym,-}\|_{L^\infty(\Omega)}} 
    =  \frac{m \otimes m}{\rho}:\left((\nabla \varphi)_{\sym} +\|(\nabla \varphi)_{\sym,-}\|_{L^\infty(\Omega)}\mathbbm{I}_d \right)\geq 0 
\end{equation*}
since both $m \otimes m$ and $\nabla \varphi +\|(\nabla \varphi)_{\sym,-}\|_{L^\infty(\Omega)}\mathbbm{I}_d $  are symmetric positive semidefinite matrices. 
Therefore, this defines a nonnegative and quadratic function in $m$, 
which is thus convex.
 
For the third contribution in~\eqref{A}, which is associated with the pressure term,
we have
 \begin{align*}
     & \rho^\gamma \mathbbm{I}_d:\nabla \varphi + \max\{\gamma-1, 1\}\|(\diiv \varphi)_{-}\|_{L^\infty(\Omega)}
     \frac{1}{\gamma -1}\rho^\gamma \\
     &\qquad= \rho^\gamma (\diiv \varphi)_+ +\rho^\gamma \Big( (\diiv \varphi)_- +  \frac{1}{\gamma -1}\max\{\gamma-1, 1\}\|(\diiv \varphi)_{-}\|_{L^\infty(\Omega)} \Big)  \geq 0,
 \end{align*}
 which corresponds to a nonnegative multiple of the convex mapping $\rho \mapsto \rho^\gamma$ in every point in $\Omega$. Therefore, this function is also convex.
 
 In order to examine the terms related to the 
 fourth contribution in~\eqref{A}, 
 we use that $h=\rho-\ol\rho$ has vanishing mean value.
 We obtain
 \[
    \begin{aligned}
    &\int_\Omega \rho \nabla \rho \cdot \nabla\diiv \varphi \dd{x}
    +2c_P\|\nabla\diiv\varphi\|_{L^\infty(\Omega)}\frac{1}{2}\|\nabla\rho\|_{L^2(\Omega)}^2
    \\
    &\qquad
    = \int_\Omega\ol\rho\nabla h\cdot  \nabla\diiv \varphi\dd{x}
    +\int_\Omega h\nabla h  \cdot  \nabla(\diiv \varphi)\dd{x}
    + c_P\|\nabla\diiv\varphi\|_{L^\infty(\Omega)}\|\nabla h\|_{L^2(\Omega)}^2.
    \end{aligned}
 \]
 While the first term is linear in $\nabla h$, the other terms define a quadratic functional in $(h,\nabla h)$,
 which is nonnegative
 due the Poincar\'e inequality~\eqref{eq:poincare}.
 Moreover, the weak lower semicontinuity follows directly from the compact embedding $H^1(\Omega)\hookrightarrow L^2(\Omega)$.
For the second last contribution in~\eqref{A},
we note that
\[
\begin{aligned}
&\frac{1}{2}|\nabla \rho|^2 \diiv \varphi +\max\{\gamma-1, 1\}\|(\diiv \varphi)_{-}\|_{L^\infty(\Omega}\frac{1}{2}|\nabla \rho|^2  \\
&\qquad
= \frac{1}{2}|\nabla \rho|^2 (\diiv \varphi)_+ 
+ \frac{1}{2}|\nabla \rho|^2 \Big( (\diiv \varphi)_- +\text{max}\{\gamma-1, 1\}\|(\diiv \varphi)_{-}\|_{L^\infty(\Omega)}\Big) 
\geq 0,
\end{aligned}
\]
which is a nonegative multiple of $|\nabla\rho|^2$ and thus convex in $\nabla\rho$.

With respect to the last term in~\eqref{A}, 
we observe that
\begin{align*}
\nabla \rho \otimes \nabla \rho :\nabla \varphi  + \frac{1}{2}|\nabla \rho|^22\|(\nabla \varphi)_{\sym,-}\|_{L^\infty(\Omega)} 
=\nabla \rho \otimes \nabla \rho : \Big( (\nabla \varphi)_{\sym}+\|(\nabla \varphi)_{\sym,-}\|_{L^\infty(\Omega)}\mathbbm{I}_d\Big)
\geq 0
\end{align*}
since $\nabla \rho \otimes \nabla \rho$ and $(\nabla \varphi)_{\text{sym}}+\|(\nabla \varphi)_{\text{sym},}\|_{L^\infty}\mathbbm{I}_d$ are symmetric positive semidefinite matrices.
Therefore, this term is quadratic and nonnegative and thus convex in $\nabla\rho$.

In conclusion, the mapping from~\eqref{eq:fct.convexified.ek} is 
convex and weakly* lower semicontinuous,
which implies the asserted convexity and semicontinuity properties.

To check the validity of the inequality \eqref{eq:AKlowerbound.ek},
we combine the previous estimates on the single contributions and obtain
by Young's inequality that
\[
\begin{aligned}
&-\langle A(\rho-\ol\rho,m),(\psi,\varphi) \rangle +\tilde{\mathcal{K}}(\varphi)\E(\rho,m)
\geq 
\int_\Omega m \cdot \nabla \psi+\ol\rho\nabla \rho\cdot  \nabla\diiv \varphi\dd{x}
\\
&\qquad
\geq -c_1\int_\Omega |m|^{\frac{2\gamma}{\gamma+1}}+|\nabla \psi|^{\frac{2\gamma}{\gamma-1}} \dd{x}
-c_2\int_\Omega \rho^\gamma +(\ol\rho |\nabla \diiv \varphi|)^{\frac{\gamma}{\gamma-1}}\dd{x}
\geq -c_3\big(\E(\rho,m)+1\big)
\end{aligned}
\]
by the coercivity estimate~\eqref{eq:Egrowth.ek},
where $c_1,c_2,c_3>0$, 
and $c_3$ only depends on $R>0$ if $\|(\psi,\varphi)\|_{\Y}\leq R$.
This shows~\eqref{eq:AKlowerbound.ek} and completes the proof.
\end{proof}

Finally, we verify that the choice of $\mathcal K$ as in~\eqref{eq:K.lsc.ek}
ensures the weak* lower semicontinuity from assumption~\ref{ass:K.lsc}.

\begin{lemma}\label{lsc ek}
    Let $\ol\rho\geq 0$.
    For each $(\psi,\varphi) \in \Y$,
    the mapping 
    \[
    (h,m) \mapsto - \langle A(h,m),(\psi,\varphi) \rangle +{\mathcal{K}}(\varphi)\ol\E(h,m),
    \]
    with $A$ and ${\mathcal K}$ as in~\eqref{A} and~\eqref{eq:K.lsc.ek},
    is weakly* lower semicontinuous in $\V^*$.
\end{lemma}

\begin{proof}
Writing again $\rho=h+\ol\rho$, 
we only have to show the weak* lower semicontinuity of the functional
\[
\rho\mapsto\int_\Omega \rho \nabla \rho \cdot \nabla(\diiv \varphi) \dd{x},
\]
defined on $L^{\gamma}(\Omega)\cap H^1(\Omega)$
since for the remaining terms we can repeat the reasoning from Lemma~\ref{convex ek}. 
However, this property follows directly from the compact embedding $H^1(\Omega) \hookrightarrow L^2(\Omega)$.
\end{proof}

After these preparations, 
we can finally prove the existence of energy-variational solutions to the Euler--Korteweg system.

\begin{proof}[Proof of Theorem~\ref{thm:existence.EulerKorteweg}]
Consider initial data $(\rho_0,m_0)$ and define 
$\ol\rho:=\frac{1}{|\Omega|}\int_\Omega\rho_0\dd{x}$.
Firstly, we observe that spaces $\V^*$ and $\Y$ from \eqref{spaces} satisfy \ref{space as}, 
and the energy $\ol\E\colon\V^* \rightarrow [0,\infty]$ from~\eqref{eq:energy.ek.hm} satisfies~\ref{energy as} due to Lemma~\ref{E coerc}.
Setting $\Psi\equiv 0$, we have verified~\ref{ass energy bal} in Lemma~\ref{lem:energybalance.ek}. 
Moreover, by Lemma~\ref{convex ek} and Lemma~\ref{lsc ek}, assumptions \ref{ass:K.lsc} and \ref{ass:K.conv} are satisfied.
In conclusion, Theorem \ref{main theorem1} implies the existence of an energy-variational solution
$(U,E)=(h,m,E)$ 
to the abstract evolution equation~\eqref{eq:general}
in the sense of Definition~\ref{sol def}
with initial value $U_0=(\rho_0-\ol\rho,m_0)$ and $E(0+)=\ol\E(\rho_0-\ol\rho,m_0)$.
Observe that we can consider test functions with values in the full vector space $\Y$
since $\dom\ol\E^*=\V^{**}=\V$ as was shown in Lemma~\ref{E coerc},
compare Remark~\ref{rem:testfunctions}.
Defining now $\rho=h+\ol\rho$,
we obtain an energy-variational solution 
to the Euler--Korteweg system~\eqref{au}
as claimed.
\end{proof}

\section{Binormal curvature flow}
\label{sec:binormal}

In this section, we apply the general theory to the binormal curvature flow,
which describes the motion of a curve in the three-dimensional space 
in the direction of the binormal vector, at a speed proportional to the curvature.
Consider a closed curve with (time-dependent) arc-length parametrization $\gamma:[0,T]\times\R \to  \R^3$,
that is, $|\partial_s\gamma(t,s)|=1$ 
and $\gamma(t,s+L)=\gamma(t,s)$ for $(t,s)\in[0,T]\times\R$, 
where $L>0$ is the length of the curve.
The curve evolves according to binormal curvature flow 
if $\gamma$ is subject to
\begin{align}\label{eq:strongBinormal}
    \t \gamma  = \partial _s\gamma  \times \partial_{ss}\gamma \quad \text{ in } [0,T]\times \R. 
\end{align}
The binormal curvature flow describes the asymptotic evolution of vortex rings according to the incompressible Euler equations. 
We refer to~\cite{BNCF1} and~\cite{BNCF2} for 
more details on this connection,
and an introduction into the general topic
and the associated generalized solution concepts.
By studying the binormal curvature flow in terms of energy-variational solutions, 
we treat them in the same solution framework as the Euler equations,
compare~\cite{envarhyp}.
Such a combination of both systems can be beneficial, 
for instance, for investigating the singular limit. 

We note that the solution framework introduced below is less regular than the one considered by Jerrad and Smets~\cite{BNCF1}.
Firstly, we consider general divergence-free measures in the context of energy-varatiational solutions, which constitute a convex superset of the integral 1-currents used in~\cite{BNCF1}. 
In other words, in the proposed framework we lose a certain level of regularity of the solution, especially the $1$-rectifiability of the curves.
Secondly, we have the restriction on the class of test functions 
since the energy functional has only linear growth,
compare Remark~\ref{rem:testfunctions}.
Omitting this restriction, we further show that the concept of energy-variational solutions is selective enough to retain a weak-strong uniqueness principle in the same spirit as in~\cite{BNCF1}.

\subsection{Energy-variational framework}
To derive a suitable energy-variational formulation of the binormal curvature flow~\eqref{eq:strongBinormal},
we first consider a weak formulation that was first introduced in~\cite{Jerrard2002vortexfilamentdyn}.

\begin{lemma}\label{lem:binormal.weakform}
    Let $\gamma\colon[0,T]\times\mathbbm T^1\to\R^3$ be a smooth solution to~\eqref{eq:strongBinormal},
    where $\mathbbm T^1:=\R/\ell\mathbbm Z$ for some $\ell>0$.
    Then for all $X\in \C^1([0,T];\C^2(\R^3;\R^3))$ it holds 
    \begin{equation}\label{eq:binormal.weak}
    \Big[\int_{\gamma_\sigma} X\cdot\tau_{\gamma_\sigma} \de\mathcal H^1\Big]\Big|_{\sigma=0}^t 
    =\int_0^t \int_{\gamma_\sigma} \partial_t X \cdot\tau_{\gamma_\sigma} \de\mathcal H^1\de\sigma
    -\int_{\gamma_\sigma}\nabla(\curl X):\tau_{\gamma_\sigma}\otimes\tau_{\gamma_\sigma}\de\mathcal H^1 \de\sigma
    \end{equation}
    for all $t\in[0,T]$, where
    $\gamma_\sigma:=\gamma(\sigma,\cdot)$ for $\sigma\in[0,T]$ and $\tau_{\gamma_\sigma}$ is the oriented tangent vector at $\gamma_\sigma$.
\end{lemma}

\begin{proof}
    The result follows directly from~\cite[Lemma 1]{BNCF1}.
\end{proof}

Since~\eqref{eq:binormal.weak} only involves first-order spatial derivatives of $\gamma$ in terms of the tangent vector $\tau_t$,
this identity may serve as the basis of a weak formulation of~\eqref{eq:strongBinormal},
compare~\cite{Jerrard2002vortexfilamentdyn,BNCF1}.
Moreover, a straightforward generalization of~\eqref{eq:binormal.weak} is to 
consider regular vector measures $\mu\in \mathcal M(\R^3;\R^3)$ instead of curves.
For those, the tangent field $\tau_\mu=\frac{\dd\mu}{\dd|\mu|}$
is defined as the Radon--Nikodym derivative of $\mu$ with respect to $|\mu|$,
and the weak form~\eqref{eq:binormal.weak} is replaced with
    \begin{align}
        \label{eq:binormal.measures}
\Big[\int_{\R^3} X \cdot \de \mu_\sigma\Big]\Big|_{\sigma=0}^t =\int_0^t \int_{\R^3} \partial_t X \cdot \de \mu_\sigma\de\sigma - \nabla (\curl X) : \tau_{\mu_\sigma}\otimes \tau_{\mu_\sigma} \de |\mu_\sigma| \de \sigma \,
    \end{align}
for $X\in \C^1([0,T];\C^2(\R^3;\R^3))$,
where $\mu_\sigma=\mu(\sigma)\in\mathcal M(\R^3;\R^3)$.
Observe that a curve pameterized by $\gamma\colon\mathbbm T^1\to\R^3$ can be identified with the Radon measure $\mu=\tau_\gamma\mathcal H^1\lfloor_{M}$,
where $M=\gamma(\mathbbm T^1)$,
that is,
\begin{equation}\label{eq:curves.as.measures}
    \forall X\in \C_0(\R^3;\R^3):\ \int_{\R^3} X \cdot\de\mu_{\gamma}
    =\int_{M} X\cdot \tau_{\gamma}\de\mathcal H^1
    =\int_{\mathbbm T^1} X(\gamma_t(s))\cdot\partial_s \gamma_t(s)\de s.
\end{equation}
In virtue of this identification, the formuation~\eqref{eq:binormal.measures} is consistent with~\eqref{eq:binormal.weak}
in the following sense.

\begin{lemma}[Consistency with smooth solutions]
    \label{lem:consistency.transport}
    Let $\gamma\colon[0,T]\times\mathbbm T^1\to\R^3$ be a smooth solution to~\eqref{eq:strongBinormal}.
    Set $M_t=\gamma_t(\mathbbm T^1)$ for $t\in[0,T]$,
    and define the measure $\mu_{\gamma_t}=\tau_{\gamma_t} \mathcal H^1\lfloor_{M_t}\in\mathcal M(\R^3;\R^3)$.    
    Then $\mu_\gamma$ fulfills~\eqref{eq:binormal.measures}
    for all $X\in \C^1([0,T];\C_c^2(\R^3;\R^3))$ and $t\in[0,T]$.
\end{lemma}

\begin{proof}
    This follows directly from Lemma~\ref{lem:binormal.weakform} and the previous considerations.
\end{proof}

The existence of weak solutions defined through~\eqref{eq:binormal.weak} or~\eqref{eq:binormal.measures} is unknown,
which motivated the study of a more generalized solution concept in~\cite{BNCF1}
based on the notion of varifolds. 
We follow here a different direction and introduce energy-variational solutions
by adapting Definition~\ref{sol def} to the current framework.

To obtain a Banach space that is suitable for our analysis,
it is natural to consider divergence-free measures $\mu\in\mathcal M(\R^3;\R^3)$.
This class is the weak* closure of the linear span
of the set of closed oriented curves, if we identify them with measures in the sense of~\eqref{eq:curves.as.measures}.
Moreover, an oriented curve is closed if and only if 
the integral over any gradient field along the curve vanishes,
which means that the associated measure is divergence free.

As the underlying space, we consider the closure of all smooth and compactly supported 
divergence-free vector fields, given by
\[
\V:=\overline{\{\varphi\in\C_c^\infty(\R^3;\R^3)\mid\diiv\varphi=0\}}^{\|\cdot\|_{\C_0(\R^3;\R^3)}}
\subset \C_0(\R^3;\R^3).
\]
Then the state space is given by the class of divergence-free Radon measures.
\begin{proposition}
The dual space of $\V$ can be identified with
\[
\V^*=\{\mu\in\mathcal M(\R^3;\R^3)\mid \diiv\mu=0 \text{ in the sense of distributions}\}.
\]
\end{proposition}
\begin{proof}
    One the one hand, it is clear that any $\mu\in\mathcal M(\R^3;\R^3)$ with $\diiv\mu=0$
    defines a linear functional on $\V$ through $\varphi\mapsto\int_{\R^3}\varphi\cdot\dd{\mu}$.
    On the other hand, let $\ell\in\V^*$, and define the class of the gradient fields by $\mathbbm G=\{\nabla\psi\mid\psi\in C_c^\infty(\R^3)\}$. Since $\mathbbm G\cap\V=\{0\}$,
    the Hahn--Banach theorem guarantees the existence of an extension $\tilde{\ell}\in\mathcal \C_0(\R^3;\R^3)^*$ such that $\tilde\ell(\varphi+\xi)=\ell(\varphi)$ 
    for all $\varphi\in\V$ and $\xi\in\mathbbm G$.
    By the Riesz representation theorem, $\tilde{\ell}$ is given by a measure $\mu$,
    and we conclude that $\mu$ is divergence-free.
\end{proof}

On $\V^*$ we define the energy 
$\mathcal{E}: \V^* \to [0,\infty]$ via 
\begin{equation}
    \label{eq:entropy.binormal}
    \mathcal{E}(\mu) 
    :=\int_{\R^3} \de |\mu| 
    =\|\mu\|_{\mathcal M(\R^3;\R^3)},
\end{equation}
and we have $\dom\E=\V^*$.
The space $\Y$ of test functions is given by 
\[
\Y : = \{\varphi\in\C_0^\infty(\R^3;\R^3)\mid \diiv\varphi=0\}
\] 
The dissipation potential is trivial, that is, $ \Psi \equiv 0$, and 
the operator $ A : \dom \E \to \Y^*$ is given by 
 \begin{align}
     \langle A (\mu) , \varphi \rangle := \int_{\R^3} \nabla (\nabla \times \f \varphi) : \tau _{\mu}\otimes \tau_{\mu} \de |\mu| \,,
 \end{align}
with tangent field $ \tau_\mu  = \frac{\de \mu}{\de |\mu|}$.
Now, we are in the position to transfer the general definition of energy-variational solutions~(Definition~\ref{sol def}) to the current context.
As is shown below, a suitable regularity weight to ensure existence is given by
\begin{equation}
\label{eq:regweigth.binormal}
\mathcal{K}(\f \varphi):= 3\| ( \nabla ( \nabla \times \f \varphi))_{\sym}\|_{L^\infty(\R^3;\R^{3\times 3 })}. 
\end{equation}

\begin{definition}\label{def:binormal}
    Let $\mu_0\in \mathcal M(\R^3;\R^3)$ with $\diiv\mu_0=0$.
    We call $ (\mu, E)\in L^\infty_{w^*} (0,T; \mathcal{M}(\R^3;\R^3))\times \BV([0,T])$ an energy-variational solution to the binormal curvature flow
    with initial value $\mu_0$ if $ E \geq \E(\mu)$ a.e.~in $(0,T)$ and if
    \begin{align}\label{envar:binormal}
    \begin{split}
                        \left [ E - \int_{\R^3} \f \varphi \cdot \de \mu  \right]\Big|_s^t + \int_s^t\int_{\R^3}  \t \f \varphi\cdot \de \mu - \int_{\R^3} \nabla (\nabla \times \f \varphi) : \tau_\mu \otimes \tau_\mu \de |\mu| 
        \de \tau 
        \\
        +\int_s^t 3 \| ( \nabla ( \nabla \times \f \varphi))_{\sym}\|_{L^\infty(\R^3;\R^{3\times 3 })} \left [ \int_{\R^3} \de |\mu| - E \right]\de \tau \leq 0
        \end{split}
\end{align}
    for all $ \f \varphi \in \C^1([0,T]; \C^2_0(\R^3;\R^3))$ with $\diiv\varphi=0$ and 
    $|\varphi|\leq 1$,
    and for a.e.~$0\leq s<t\leq T$, including $s =0 $ with $ \mu(0) = \mu_0$. 
\end{definition}

\begin{remark}
    \label{rem:testfct.binormal}
    We only consider test functions $\varphi$ with $|\varphi|\leq 1$ in Definition~\ref{def:binormal}.
    This constraint corresponds to the restriction in Definition~\ref{sol def}
    that the test functions take values in $\Y\cap\dom\partial\E^*$,
    see Lemma~\ref{lem:energy.binormal} below.
    Further note that the formulation~\eqref{envar:binormal} is invariant under addition of gradient fields $\nabla\psi$ to the test functions
    since $\mu$ is divergence-free and $\curl\nabla\psi=0$,
    so that a larger class of functions would be admissible.
    For instance, one may consider compactly supported test functions 
    $\varphi\in\C^1([0,T];\C^\infty_c(\R^3;\R^3))$ 
    with Helmholtz decomposition $\varphi=\curl\chi+\nabla\psi$
    such that $|\curl\chi|\leq 1$.
    Note that the Helmholtz decomposition is indeed available for $\varphi\in\C^\infty_c(\R^3;\R^3)$,
    and one can construct $\chi\in\C_0^\infty(\R^3;\R^3)$ and $\psi\in\C_0^\infty(\R^3)$ by the vector potentials
    \[
    \chi(x)=\frac{1}{4\pi}\int_{\R^3}\frac{\curl\varphi(y)}{|x-y|}\dd{y},
    \qquad
    \psi(x)=-\frac{1}{4\pi}\int_{\R^3}\frac{\diiv\varphi(y)}{|x-y|}\dd{y}.
    \]
\end{remark}

\subsection{Existence of energy-variational solutions}

The abstract theory developed before leads to the following existence theorem.

\begin{theorem}\label{thm:binormalex}
Let $ \mu_0 \in \V^*$ such that $\E(\mu_0)<\infty$. Then there exists an energy-variational solution in the sense of Definition~\ref{def:binormal} with  $E(0)= \int_{\R^3}\de |\mu_0|$.  
\end{theorem}

In order to apply the general result from Theorem~\ref{main theorem1},
we verify necessary assumptions~\ref{space as}--\ref{ass:K.conv}
in the following two lemmata. 

\begin{lemma}[Properties of the energy functional]
    \label{lem:energy.binormal}
    Let the functional $ \mathcal{E}:\V^* \to [0,\infty]$ be defined by~\eqref{eq:entropy.binormal}.
    Then $\E$ is convex and weakly* lower semicontinuous and trivially satisfies 
    $\E(\mu)\geq\|\mu\|_{\V^*}$.
    Moreover, for all $\xi\in\V^{**}$ it holds
    \begin{equation}
        \label{Estar.binormal}
        \E^*(\xi)=
        \begin{cases}
            0 & \text{if }\|\xi\|_{\V^{**}}\leq 1,
            \\
            \infty &\text{else.}
        \end{cases}
    \end{equation}
    In particular, $\dom\E^*=\dom\partial\E^*$,
    and 
    $0\in\partial\E^*(\xi)$ for all $\xi\in\dom\partial\E^*$,
    so that
    for any $\varphi \in \Y\cap\dom\partial\E^*$ there exists an element 
    $\mu \in \partial \E^*(\f\varphi)$ (namely $\mu=0$) such that 
    $\langle A (\mu),\varphi\rangle= 0$. 
\end{lemma}

\begin{proof}
    As $\E=\|\cdot\|_{\mathcal M(\R^3;\R^3)}$ on the weakly* closed subspace $\V^*$,
    the convexity and weak* lower semicontinuity follow directly,
    and the coercivity estimate is trivial.
    The identification of $\E^*$ 
    is a standard example from calculus of variations,
    and~\eqref{Estar.binormal} directly implies the remaining assertions.
\end{proof}

To show that the choice of the regularity weight $\mathcal K$ as in~\eqref{eq:regweigth.binormal}
is suitable, we use the following convexity result.

\begin{lemma}
    \label{lem:convexity.1hom}
    Let $\mathbbm M\in\R^{d\times d}_{\sym}$, $d\in\N$, be a symmetric matrix, 
    and let $\lambda_m,\lambda_M\in\R$ be the smallest and the largest eigenvalue of $\mathbbm M$, respectively.
    For $\gamma\in\R$ consider the function
    \[
    f\colon\R^d\to\R,\quad f(\xi)=\gamma|\xi|+ \frac{\mathbbm M\xi\cdot\xi}{|\xi|},
    \]
    where we set $f(0)=0$.
    Then $f$ is convex if and only if $\gamma\geq \lambda_M-2\lambda_m$.
\end{lemma}

\begin{proof}
    We follow the proof of~\cite[Theorem 1.2]{DacorognaHaeberly1996convhomog},
    where a similar function on the space $\R^{2\times2}$ was considered.
    To study convexity, we first compute the second-order derivatives of $f$.
    As $f$ is continuous, it suffices to consider $\xi\neq 0$.
    We have
    \[
    \begin{aligned}
    \partial_j f(\xi)
    &=\gamma \frac{\xi_j}{|\xi|}+\frac{2(\mathbbm M\xi)_j}{|\xi|}-\frac{(\mathbbm M\xi\cdot\xi)\xi_j}{|\xi|^3},
    \\
    \partial_j\partial_k f(\xi)
    &=\gamma\frac{\delta_{jk}}{|\xi|}-\gamma\frac{\xi_j\xi_k}{|\xi|^3}
    +\frac{2\mathbbm M_{jk}}{|\xi|}-\frac{4(\mathbbm M\xi)_j\xi_k+(\mathbbm M\xi\cdot\xi)\delta_{jk}}{|\xi|^3}
    +\frac{3(\mathbbm M\xi\cdot\xi)\xi_j\xi_k}{|\xi|^5}.
    \end{aligned}
    \]
    Hence, $f$ is convex if and only if for all $\xi\neq 0$ and $\mu\in\R^d$ the quantity
    \[
    \begin{aligned}
    \sum_{j,k=1}^d\partial_j\partial_k f(\xi)\mu_j\mu_k
    &=\gamma\Big(\frac{|\mu|^2}{|\xi|}-\frac{(\xi\cdot\mu)^2}{|\xi|^2}\Big)
    \\
    &+\frac{1}{|\xi|^5}\Big(
    2|\xi|^4(\mathbbm M\mu\cdot\mu)-4|\xi|^2(\mathbbm M\xi\cdot\mu)(\xi\cdot\mu)
    -|\xi|^2(\mathbbm M\xi\cdot\xi)|\mu|^2+3(\mathbbm M\xi\cdot\xi)(\xi\cdot\mu)^2
    \Big)
    \end{aligned}
    \]
    is nonnegative.
    As $\nabla^2 f$ is homogeneous of degree $-1$, it suffices to consider $|\xi|=1$.
    We decompose $\mu=\alpha\xi+\beta\eta$ with $s,t\in\R$ and $|\eta|=1$ such that $\xi\cdot\eta=0$.
    Then it holds
    \[
    \begin{aligned}
    &|\mu|^2=\alpha^2+\beta^2, \quad
    \xi\cdot\mu=\alpha, \quad
    \mathbbm M\xi\cdot\mu= \alpha (\mathbbm M\xi\cdot\xi)+\beta (\mathbbm M\xi\cdot\eta),
    \\
    &M\mu\cdot\mu= \alpha^2 (\mathbbm M\xi\cdot\xi)+2\alpha\beta (\mathbbm M\xi\cdot\eta)+\beta^2(\mathbbm M\eta\cdot\eta),
    \end{aligned}
    \]
    and we have
    \[
    \sum_{j,k=1}^d\partial_j\partial_k f(\xi)\mu_j\mu_k
    =\beta^2\big(\gamma+2\mathbbm M\eta\cdot\eta-\mathbbm M\xi\cdot\xi\big).
    \]
    Therefore, $f$ is convex if and only if 
    \[
    \gamma+2\inf_{|\eta|=1}\mathbbm M\eta\cdot\eta - \sup_{|\xi|=1}\mathbbm M\xi\cdot\xi \geq 0.
    \]
    As $\mathbbm M$ is symmetric, this corresponds to the asserted characterization.
\end{proof}

\begin{lemma}[Regularity weight]\label{lem:bilow}
    For every $\varphi\in\Y$, the mapping $\mu \mapsto - \langle A(\mu),\f \varphi \rangle + \mathcal{K}(\f \varphi) \E(\mu)$
    is nonnegative, convex and weakly* lower semicontinuous in $\V^*$.
\end{lemma}

\begin{proof}
    We rewrite the mapping as
    \[
    \begin{aligned}
    \mu \mapsto
    \int_{\R^3} -\nabla (\nabla \times \f \varphi) : \tau_\mu \otimes \tau_\mu 
        + 3 \| ( \nabla ( \nabla \times \f \varphi))_{\sym}\|_{L^\infty(\R^3;\R^{3\times 3 })} \de |\mu|
    =\int_{\R^3} f(x,\tau_\mu(x))\de|\mu|(x)
    \end{aligned}
    \]
    with
    \[
    f(x,\xi)
    :=-\frac{[\nabla (\curl \varphi)(x)]_{\sym} : \xi \otimes \xi}{|\xi|}
    + 3 \| ( \nabla ( \curl \varphi))_{\sym}\|_{L^\infty(\R^3;\R^{3\times 3 })} |\xi|.
    \]
    For this identity, we use that $|\tau_\mu|=1$ and that $\tau_\mu\otimes\tau_\mu$ is symmetric.
    The function $f$ is continuous, and it is $1$-homogeneous with respect to $\xi$. 
    Moreover, $f(x,\cdot)$ is convex by Lemma~\ref{lem:convexity.1hom}.
    By a version of Reshetnyak's semicontinuity theorem, 
    see~\cite[Proposition 2.37]{AmbrosioFuscoPallara2000} for instance,
    this implies the asserted convexity and weak* lower semicontinuity.
    The nonnegativity trivial.
\end{proof}

We can now conclude the existence of energy-variational solutions to the binormal curvature flow 
from Theorem~\ref{main theorem1}.

\begin{proof}[Proof of Theorem~\ref{thm:binormalex}]
    In order to apply the general existence theorem (Theorem~\ref{main theorem1}) for this setting,
    we verify the assumptions~\ref{space as}--\ref{ass:K.conv}.
    These follow from the previous lemmata with the choices $\Psi=0$ and $\mathcal K=\tilde{\mathcal K}$.
\end{proof}

\subsection{Weak-strong uniqueness principle}

Although the concept of energy-variational solutions 
seems to be less restrictive than the generalized solutions from~\cite{BNCF1,BNCF2}, we can still adapt the  for weak-strong uniqueness from~\cite[Theorem~2]{BNCF1} to our setting, making it a reasonable concept. 
However, we have to slightly modify the definition of energy-variational solutions
by omitting the norm restriction on the admissible test functions
that appears in Definition~\ref{def:binormal}.

Recall the identification of a curve $\gamma$
with a measure $\mu_\gamma$ from~\eqref{eq:curves.as.measures}.

\begin{theorem}[Weak-strong uniqueness] \label{thm:weakstrong}
Let $L>0$ and $\gamma: [0,T] \times \mathbbm T^1 \rightarrow \mathbbm{R}^3$ denote a smooth classical solution of the binormal curvature flow without self-intersection,
that is, the arc-length parameterization $\gamma_t:=\gamma(t, \cdot)$ solves~\eqref{eq:strongBinormal} and is injective for any $t\in[0,T]$.
Let $(\mu,E)$ be an energy-variational solution to the binormal curvature flow on $[0,T]$ 
with $E(0+)=\E(\mu_0)$ and
such that~\eqref{envar:binormal} holds
for all $\varphi \in \C^1([0,T]; \C^2_c(\R^3;\R^3))$ and for a.e.~$0\leq s<t\leq T$.
If the initial values of $\gamma$ and $\mu$ coincide
in the sense that $\mu_{\gamma_0}=\mu_0$,
then $\mu_{\gamma_t}=\mu_t$ for a.a~$t\in[0,T]$.
\end{theorem}

To show this result, we strongly follow
the proof of \cite[Theorem~2]{BNCF1}. We only need to adapt the concluding argument a little. 
As one ingredient, 
we use the following elementary result
on divergence-free vector measures
supported on a curve.

\begin{lemma}\label{lem:const}
Let $M\subset\mathbbm R^3$ be a closed, connected $\C^1$-curve 
with arc-length parametrization $\gamma:\mathbbm T^1\to M$.
Let
\[
\mu=\theta\,\tau_\gamma\,\mathcal H^1\lfloor_M
\]
with $\theta\in L^1(M)$. If $\operatorname{div}\mu=0$ in $\mathcal D'(\mathbbm R^3)$, then $\theta$ is constant a.e.\ on $M$.
\end{lemma}

\begin{proof}
For any $\varphi\in \C_c^\infty(\mathbbm R^3)$ the assumption $\diiv\mu=0$ yields
\[
0=\langle\operatorname{div}\mu,\varphi\rangle
=-\langle\mu,\nabla\varphi\rangle
=-\int_M \theta(x)\,\tau_\gamma(x)\cdot\nabla\varphi(x)\,\dd\mathcal H^1(x).
\]
With the arc-length parameter $s\in\mathbbm T^1$, such that $x=\gamma(s)$ and $\dd \mathcal H^1(x)=\dd s$, this gives
\[
0=-\int_{\mathbbm T^1} \theta(\gamma(s))\,\frac{\dd}{\dd s}\big(\varphi(\gamma(s))\big)\,\dd s
=\int_{\mathbbm T^1} \frac{\dd}{\dd s}\theta(\gamma(s))\,\varphi(\gamma(s))\,\dd s
\]
by integration by parts.
Thus $(\theta \circ \gamma)'=0$ in $\mathcal D'(\mathbbm T^1)$, whence $\theta$ is constant a.e.~on $M$.
\end{proof}

We further introduce some tools from~\cite{BNCF1}. 
For $\gamma=(\gamma_t)_{t\in[0,T]}$ as in Theorem~\ref{thm:weakstrong},
we set
$$
r \equiv 
r(\gamma):=\frac{1}{2} \min _{t \in [0,T]} \min \left(\left\|\partial_{s s} \gamma_t\right\|_{\infty}^{-1}, r_s(t)\right)>0.
$$
Here, the security radius $r_s(t)$, $t\in[0,T]$, is defined as the largest positive real number with the property that every point $x\in\R^3$ with $\operatorname{dist}\left(x, \gamma(t)\right)<r_s(t)$ has a unique closest point $P_t(x)$ on $\gamma_t$. 
Indeed, such a point always exists and it holds $r(t)>0$ as $\gamma_t\in \C^1(\R)$
has no self-intersections.
Define the vector field $X_{\gamma}$ on $\mathbbm{R}^3 \times [0,T]$ by 
\begin{equation}\label{Xgamma}
    X_{\gamma}(t,x)=f\left(\operatorname{dist}\left(x, \gamma_t\right)^2\right) \tau_t\left(P_t(x)\right),
\end{equation}
where $\tau_t$ is the oriented unit tangent vector along $\gamma_t$ and 
$$
f\left(d^2\right)= 
\begin{cases}
\left(1-(d / r)^2\right)^3 & \text { for } 0 \leq d^2 \leq r^2, \\ 
0 & \text { for } d^2 \geq r^2.
\end{cases}
$$
The derivation of the following proposition is technical and due to~\cite[Proposition~4]{BNCF1}.

\begin{proposition}\label{prop:1}
    Let $\gamma$ be as in Theorem~\ref{thm:weakstrong}, and let $X_{\gamma}$ be given by~\eqref{Xgamma}.
    For any $\xi_0 \in \mathbbm S^2 \subset \mathbbm{R}^3$, the estimate
$$
\left|\partial_t X_{\gamma} \cdot \xi_0-\nabla\left(\curl X_{\gamma}\right):\left(\xi_0 \otimes \xi_0\right)\right| \leq K\left(1-X_{\gamma} \cdot \xi_0\right)
$$
holds on $[0,T^*]\times\mathbbm{R}^3$, where
$$
K:=\frac{54}{r^2}+14\left\|\partial_{s s s} \gamma\right\|_{L^{\infty}\left([0,T^*] \times \mathbbm T^1\right)}.
$$
\end{proposition}
\begin{proof}
    See~\cite[Proposition~4]{BNCF1}.
\end{proof}

Now we have provided all the necessary ingredients to
deduce Theorem~\ref{thm:weakstrong}.

\begin{proof}[Proof of Theorem~\ref{thm:weakstrong}]
For $\varphi=X_{\gamma}$ as defined in~\eqref{Xgamma}, 
we obtain from~\eqref{envar:binormal} that
\begin{align*}
    \left[ E - \E(\mu) + \int_{\R^3}\left( 1 -  X_{\gamma} \cdot \frac{\de \mu}{\de |\mu|}\right) \de |\mu|\right] \Big |_0^t + \int_0^t \int_{\R^3} \t X_{\gamma}   \cdot \frac{\de\mu}{\de |\mu|} -  \nabla (\curl X_{\gamma}) : \frac{\de\mu}{\de |\mu|} \otimes \frac{\de\mu}{\de |\mu|}  \de |\mu| \de \sigma 
    \\
    + \int_s^t \| ( \nabla ( \curl X_{\gamma}))_{\sym,+}\|_{L^\infty(\R^3;\R^{3\times 3 })} \left [ \int_{\R^3} \de |\mu| - E \right]\de \sigma \leq 0\,.
\end{align*}
    Invoking Proposition~\ref{prop:1}, we infer the inequality 
    \begin{equation*}
        \left[ E - \E(\mu) + \int_{\R^3}\left( 1 -  X_{\gamma} \cdot \frac{\de \mu}{\de |\mu|}\right) \de |\mu|\right] \Big |_0^t \leq \int_0^tC(X) \left[ E - \E(\mu) + \int_{\R^3}\left( 1 -  X_{\gamma} \cdot \frac{\de \mu}{\de |\mu|}\right) \de |\mu|\right] \de \sigma \,,
    \end{equation*}
    where $C(X_{\gamma}):=\max\left\{K,\| ( \nabla ( \curl X_{\gamma}))_{\sym,+}\|_{L^\infty(\R^3;\R^{3\times 3 })}\right\}$. 
    Since $\mu_{\gamma_0}=\mu_0$ implies $\tau_{0}=\tau_{\mu_0}$,
    and since $E(0+)=\E(\mu_0)$,
    Gronwall's inequality lets us conclude 
    \begin{equation}
        \label{eq:binorm.weakstrong.gronwall}
         E = \E(\mu) \quad \text{and}\quad  \int_{\R^3}\left( 1 -  X_{\gamma} \cdot \frac{\de \mu}{\de |\mu|}\right) \de |\mu| = 0 
    \end{equation}
    a.e.~in $[0,T]$. 
    In particular, for a.e.~$t\in[0,T]$, 
    the second identity implies $X_{\gamma}(t,\cdot) = \frac{\de \mu_t}{\de |\mu_t|}$ $|\mu_t|$-a.e.,
    and the construction of $X_{\gamma} $ and the property $\big|\frac{\de \mu_t}{\de |\mu_t|}\big|=1$ 
    show that $|\mu_t|$ is absolutely continuous with respect to $|\mu_{\gamma_t}|$,
    and that
    $ \frac{\de \mu_t}{\de |\mu_t|}= \tau_{t} $ $|\mu_t|$-a.e,
    for $\tau_{t}$ as in~\eqref{Xgamma}.
    Let $M_t=\gamma(\mathbbm T^1)$ be the support of the curve $\gamma_t$,
    and set $M=\bigcup_{t\in[0,T]}\{t\}\times M_t$.
    If we consider $\mu=(\mu_t)_{t\in[0,T]}$ and $\mu_\gamma=(\mu_{\gamma_t})_{t\in[0,T]}$ as elements of $\mathcal M([0,T]\times\R^3;\R^3)$,
    the Radon-Nikodym theorem yields the decomposition
    $\dd\mu=\frac{\dd\mu}{\dd|\mu|}\dd|\mu|=\tau\,\theta\,\dd|\mu_\gamma|$,
    with $\theta=\frac{\dd|\mu|}{\dd|\mu_\gamma|} \in L^1(M)$, $\theta\geq0$.
    This implies
    $\dd\mu_t = \theta_t \dd\mu_{\gamma_t}=\theta_t\,\tau_t \dd\mathcal H^1\lfloor_{M_t}$ for $\theta_t=\theta(t,\cdot)$.
    As $\mu_t$ is divergence-free for a.a.~$t\in[0,T]$, Lemma~\ref{lem:const} implies that $ \theta_t$ is constant on $M_t$.
    Moreover, we have $\theta_0=1$ since $\mu_0=\mu_{\gamma_0}$.

It remains to show $\theta\equiv 1$.
Since $ E=\E(\mu)$ by~\eqref{eq:binorm.weakstrong.gronwall},
the defect term in the energy-variational formulation~\eqref{envar:binormal} vanishes.
As we omitted the restriction $|\varphi|\leq 1$,
we can choose $\varphi = \frac{X}{\alpha}$ with $\alpha>0$ and $X\in \C^1([0,t^*];\C^2(\R^3;\R^3))$
in~\eqref{envar:binormal},
multiply the resulting inequality by $\alpha$ and pass to the limit with $\alpha \searrow 0$.
Then we see that $\mu$ satisfies the generalized formulation~\eqref{eq:binormal.measures}.
Actually, we first obtain~\eqref{eq:binormal.measures} with inequality sign, 
but the linear structure of the class of test functions leads to an equality. 
With $X=X_{\gamma}$ from~\eqref{Xgamma}, this yields 
\begin{align}\label{inserting}
\begin{split}
         (\theta_t-1)L(\gamma)
         &=\bigg[  \int_{\mathbbm T^1} \theta_\sigma\circ\gamma  \de s \bigg] \bigg|_{\sigma=0}^t 
         =  \left[ \int_{\R^3} \theta_\sigma X_{\gamma}(\sigma) \cdot \tau_\sigma \de \mathcal{H}^1\lfloor_{M_\sigma} \right] \Big|_{\sigma=0}^t 
         \\& = \int_0^t \theta_\sigma \int_{\R^3} \left[ \t X_{\gamma}(\sigma)\cdot \tau_{\sigma} -\nabla (\curl X_{\gamma})(\sigma) : \tau_{\sigma} \otimes \tau_{\sigma} \right]\de \mathcal{H}^1 \lfloor_{M_\sigma}  \de \sigma 
\end{split}
\end{align}
for any $t>0$.
To simplify the right-hand side, observe that
for any smooth vector field $X$, it holds 
$$
\begin{aligned}
    \big[\nabla (\curl X) : \tau_{\gamma} \otimes \tau_{\gamma} \big]\circ\gamma
=\epsilon_{i j k}\left(\left(\partial_{i l} X^j\right) \circ \gamma\right) \partial_s \gamma^l \partial_s \gamma^k
=\epsilon_{i j k} \partial_s\left(\left(\partial_i X^j\right) \circ \gamma\right) \partial_s \gamma^k ,
\end{aligned}
$$
where $\epsilon_{i j k}$ is the usual Levi--Civita symbol, 
and summation over repeated indices is implicit. 
Moreover, since $\gamma$ is a classical solution to the binormal curvature flow~\eqref{eq:strongBinormal}, standard vector calculus implies 
$\partial_t\gamma\times\partial_s\gamma=\partial_{ss}\gamma$.
By integration by parts, for any $\sigma\in[0,T]$ we thus obtain 
\begin{align*}
    \int_{\R^3}\left[ \nabla (\curl X) : \tau_{\sigma} \otimes \tau_{\sigma} \right]\de \mathcal{H}^1\lfloor_{M_\sigma} 
    &= - \int_{\mathbbm T^1} \epsilon_{i j k} \left(\partial_i X^j \circ \gamma \right) \partial_{ss} \gamma^k \de s 
    \\
    & = - \int_{\mathbbm T^1} \epsilon_{i j k} \left(\partial_i X^j \circ \gamma\right) ( \partial_t \gamma\times  \partial_s \gamma)^k \de s
     \\
    & = - \int_{\mathbbm T^1} \epsilon_{i j k} \left(\partial_i X^j \circ \gamma \right) \epsilon_{klm} ( \partial_t \gamma^l  \partial_s \gamma^m)   \de s
    \\
     & = - \int_{\mathbbm T^1} (\delta_{il}\delta_{jm}-\delta_{im}\delta_{jl}) \left(\partial_i X^j \circ \gamma \right) ( \partial_t \gamma ^l \partial_s \gamma^m)  \de s
     \\
     & = - \int_{\mathbbm T^1}  \left(\partial_l X^m \circ \gamma \right) ( \partial_t \gamma^l  \partial_s \gamma^m)-\left(\partial_m X^l \circ \gamma \right) ( \partial_t \gamma^l  \partial_s \gamma^m)   \de s
     \\
     & = - \int_{\mathbbm T^1}  \partial _t \left( X^m \circ \gamma \right)   \partial_s \gamma^m- \partial_s \left( X^l \circ \gamma \right)  \partial_t \gamma^l    \de s
     \\
     & = - \int_{\mathbbm T^1}  \partial _t \left( X^m \circ \gamma \right)   \partial_s \gamma^m+  \left( X^l \circ \gamma \right)  \partial_t  \partial_s \gamma^l    \de s\,,
\end{align*}
where we simply write $\gamma=\gamma_\sigma$.
Choosing $ X = X_{\gamma}$, we find that both terms in the last line coincide on $\mathbbm T^1$ and
$$ 
 \int_{\R^3}\left[ \nabla (\curl X_{\gamma}) : \tau_{\sigma} \otimes \tau_{\sigma} \right]\de \mathcal{H}^1\lfloor_{M_\sigma}   = - \frac{\de}{\de t} \int_{\mathbbm T^1} | \partial_s \gamma|^2 \de s = 0
$$
for all $\sigma\in[0,T]$ as $ \gamma_\sigma$ is parameterized by arc length. 
Moreover,
it holds $\t X _{\gamma}(\sigma)\cdot \tau_{\sigma} =0$ on $M_\sigma$
since $ X _{\gamma}\circ \gamma = \partial _s \gamma$ implies
$$
[\t X _{\gamma}(\sigma)\cdot \tau_{\sigma}]\circ\gamma = \t\partial _s  \gamma_\sigma \cdot \partial _s  \gamma_\sigma = \frac{1}{2} \partial_t  | \partial _s  \gamma_\sigma  |^2  = 0.
$$
In conclusion, the right-hand side of~\eqref{inserting} vanishes,
and we obtain $\theta= 1$ a.e.~in $[0,T]$, and thus the assertion. 
\end{proof}

\begin{remark}
    In Theorem~\ref{thm:weakstrong},
    we had to assume that the energy-variational solution $(\mu,E)$ satisfies~\eqref{envar:binormal}
    for all $\varphi \in \C^1([0,T]; \C^2(\R^3;\R^3))$,
    without the restriction $|\varphi|\leq 1$,
    which is present in Definition~\ref{def:binormal}.
    This was necessary to conclude that $\mu$ satisfies the generalized formulation~\eqref{eq:binormal.measures},
    which leads to $\theta\equiv 1$ in the previous proof.
    Without this refinement, one can merely conclude $\theta\leq 1$ since~\eqref{envar:binormal} with $\varphi=0$ and~\eqref{eq:binorm.weakstrong.gronwall}
    imply that $\E(\mu)$ is nonincreasing.
    However, the existence of an energy-variational solution of this type 
    remains open and does not follow from the abstract theory developed before.
\end{remark}

\subsection*{Acknowledgments}
This research has been 
funded and supported by Deutsche Forschungsgemeinschaft (DFG) 
through grant SPP 2410 ``Hyperbolic Balance Laws in Fluid Mechanics:~Complexity, Scales, Randomness (CoScaRa)'', Project Number
525941602.
T.~Eiter further acknowledges the funding by Deutsche
Forschungsgemeinschaft (DFG) through grant CRC 1114 ``Scaling Cascades in
Complex Systems'', Project Number 235221301, Project YIP.

\printbibliography

@article{giesselmannlattanziotzavaras2017eulerkorteweg,
 author = {Giesselmann, Jan and Lattanzio, Corrado and Tzavaras, Athanasios E.},
 title = {Relative energy for the {Korteweg} theory and related {Hamiltonian} flows in gas dynamics},
 fjournal = {Archive for Rational Mechanics and Analysis},
 journal = {Arch. Ration. Mech. Anal.},
 volume = {223},
 number = {3},
 pages = {1427--1484},
 year = {2017},
 doi = {10.1007/s00205-016-1063-2},
}

@article{audiardhaspot2017eulerkorteweg,
 author = {Audiard, Corentin and Haspot, Boris},
 title = {Global well-posedness of the {Euler}-{Korteweg} system for small irrotational data},
 fjournal = {Communications in Mathematical Physics},
 journal = {Commun. Math. Phys.},
 volume = {351},
 number = {1},
 pages = {201--247},
 year = {2017},
 doi = {10.1007/s00220-017-2843-8},
}

@article{benzonigavagedanchindescombes2007eulerkorteweg,
 author = {Benzoni-Gavage, Sylvie and Danchin, Raphaël and Descombes, Stéphane},
 title = {On the well-posedness for the {Euler}-{Korteweg} model in several space dimensions},
 fjournal = {Indiana University Mathematics Journal},
 journal = {Indiana Univ. Math. J.},
 volume = {56},
 number = {4},
 pages = {1499--1579},
 year = {2007},
 doi = {10.1512/iumj.2007.56.2974},
}

@article{dafermos1973entropyratecriterion,
 author = {Dafermos, Constantine M.},
 title = {The entropy rate admissibility criterion for solutions of hyperbolic conservation laws},
 fjournal = {Journal of Differential Equations},
 journal = {J. Differ. Equations},
 volume = {14},
 pages = {202--212},
 year = {1973},
 doi = {10.1016/0022-0396(73)90043-0},
}

@article{breitfeireslhofmanova2020semiflowisentropicEuler,
 author = {Breit, Dominic and Feireisl, Eduard and Hofmanov{\'a}, Martina},
 title = {Solution semiflow to the isentropic {Euler} system},
 fjournal = {Archive for Rational Mechanics and Analysis},
 journal = {Arch. Ration. Mech. Anal.},
 volume = {235},
 number = {1},
 pages = {167--194},
 year = {2020},
 doi = {10.1007/s00205-019-01420-6},
}

@article{breitfeireislhofmanova2020semiflowcompleteEuler,
 author = {Breit, Dominic and Feireisl, Eduard and Hofmanov{\'a}, Martina},
 title = {Dissipative solutions and semiflow selection for the complete {Euler} system},
 fjournal = {Communications in Mathematical Physics},
 journal = {Commun. Math. Phys.},
 volume = {376},
 number = {2},
 pages = {1471--1497},
 year = {2020},
 doi = {10.1007/s00220-019-03662-7},
}

@misc{eiterschindler2025timeasymptoticselfsimilarity,
      title={Time-asymptotic self-similarity of the damped compressible Euler equations in parabolic scaling variables}, 
      author={Thomas Eiter and Stefanie Schindler},
      year={2025},
      eprint={2507.03688},
      archivePrefix={arXiv},
}

@misc{houwangyang2025nonuniquenesslerayhopfsolutionsunforced,
      title={Nonuniqueness of Leray-Hopf solutions to the unforced incompressible 3D Navier-Stokes Equation}, 
      author={Thomas Hou and Yixuan Wang and Changhe Yang},
      year={2025},
      eprint={2509.25116},
      archivePrefix={arXiv},
}

@book{pedregal1997parametrizedmeasures,
 author = {Pedregal, Pablo},
 title = {Parametrized measures and variational principles},
 fseries = {Progress in Nonlinear Differential Equations and Their Applications},
 series = {Prog. Nonlinear Differ. Equ. Appl.},
 volume = {30},
 year = {1997},
 publisher = {Basel: Birkh{\"a}user},
}

@book{grisvard1985elliptic,
 author = {Grisvard, P.},
 title = {Elliptic problems in nonsmooth domains},
 fseries = {Monographs and Studies in Mathematics},
 series = {Monogr. Stud. Math.},
 volume = {24},
 year = {1985},
 publisher = {Pitman, Boston, MA},
}

@book{mielkeroubicek2015ris,
 author = {Mielke, Alexander and Roub{\'{\i}}{\v{c}}ek, Tom{\'a}{\v{s}}},
 title = {Rate-independent systems. {Theory} and application},
 fseries = {Applied Mathematical Sciences},
 series = {Appl. Math. Sci.},
 volume = {193},
 year = {2015},
 publisher = {New York, NY: Springer},
 doi = {10.1007/978-1-4939-2706-7},
}

@book{roubicek2005nonlpde,
 author = {Roub{\'{\i}}{\v{c}}ek, Tom{\'a}{\v{s}}},
 title = {Nonlinear partial differential equations with applications},
 fseries = {ISNM. International Series of Numerical Mathematics},
 series = {ISNM, Int. Ser. Numer. Math.},
 volume = {153},
 year = {2005},
 publisher = {Basel: Birkh{\"a}user},
 zbMATH = {2200202},
 Zbl = {1087.35002}
}

@misc{eiter2025viscoelastoplasticity,
      title={Solution concepts for a model of visco-elasto-plasticity with slight compressibility}, 
      author={Thomas Eiter},
      year={2025},
      eprint={2512.17464},
      archivePrefix={arXiv},
}

@article{DacorognaHaeberly1996convhomog,
 author = {Dacorogna, Bernard and Haeberly, Jean-Pierre},
 title = {On convexity properties of homogeneous functions of degree one},
 fjournal = {Proceedings of the Royal Society of Edinburgh. Section A. Mathematics},
 journal = {Proc. R. Soc. Edinb., Sect. A, Math.},
 volume = {126},
 number = {5},
 pages = {947--956},
 year = {1996},
 doi = {10.1017/S0308210500023180},
}

@book{AmbrosioFuscoPallara2000,
 author = {Ambrosio, Luigi and Fusco, Nicola and Pallara, Diego},
 title = {Functions of bounded variation and free discontinuity problems},
 fseries = {Oxford Mathematical Monographs},
 series = {Oxford Math. Monogr.},
 year = {2000},
 publisher = {Oxford: Clarendon Press},
}

@article{Jerrard2002vortexfilamentdyn,
 author = {Jerrard, Robert L.},
 title = {Vortex filament dynamics for {Gross}-{Pitaevsky} type equations},
 fjournal = {Annali della Scuola Normale Superiore di Pisa. Classe di Scienze. Serie V},
 journal = {Ann. Sc. Norm. Super. Pisa, Cl. Sci. (5)},
 volume = {1},
 number = {4},
 pages = {733--768},
 year = {2002},
 url = {https://eudml.org/doc/84485},
}

@article{diss19,
 author = {Lasarzik, Robert},
 title = {Dissipative solution to the {Ericksen}-{Leslie} system equipped with the {Oseen}-{Frank} energy},
 fjournal = {ZAMP. Zeitschrift f{\"u}r angewandte Mathematik und Physik},
 journal = {Z. Angew. Math. Phys.},
 volume = {70},
 number = {1},
 pages = {39},
 note = {Id/No 8},
 year = {2019},
 doi = {10.1007/s00033-018-1053-3},
}

@Book{barbu,
 Author = {Barbu, Viorel and Precupanu, Teodor},
 Title = {Convexity and optimization in {Banach} spaces.},
 Edition = {4th updated and revised ed.},
 FSeries = {Springer Monographs in Mathematics},
 Series = {Springer Monogr. Math.},
 Year = {2012},
 Publisher = {Dordrecht: Springer},
 DOI = {10.1007/978-94-007-2247-7},
}

@article{BNCF1,
 author = {Jerrard, Robert L. and Smets, Didier},
 title = {On the motion of a curve by its binormal curvature},
 fjournal = {Journal of the European Mathematical Society (JEMS)},
 journal = {J. Eur. Math. Soc. (JEMS)},
 volume = {17},
 number = {6},
 pages = {1487--1515},
 year = {2015},
 doi = {10.4171/JEMS/536},
}

@article{BNCF2,
 author = {Jerrard, Robert L. and Seis, Christian},
 title = {On the vortex filament conjecture for {Euler} flows},
 fjournal = {Archive for Rational Mechanics and Analysis},
 journal = {Arch. Ration. Mech. Anal.},
 volume = {224},
 number = {1},
 pages = {135--172},
 year = {2017},
 doi = {10.1007/s00205-016-1070-3},
}

@Book{troltzsch,
Author = {Fredi Troltzsch},
Title = {Optimal Control of Patrial Differential Equations},
Publisher = {A},
Year = {2009},
}

@Article{envarhyp,
 Author = {Eiter, Thomas and Lasarzik, Robert},
 Title = {Existence of energy-variational solutions to hyperbolic conservation laws},
 FJournal = {Calculus of Variations and Partial Differential Equations},
 Journal = {Calc. Var. Partial Differ. Equ.},
 Volume = {63},
 Number = {4},
 Pages = {40},
 Note = {Id/No 103},
 Year = {2024},
 DOI = {10.1007/s00526-024-02713-9},
}

@book{relax,
    author ={Tomas Roubicek} ,
    title = {Relaxation in Optimization Theory and Variational Calculus.},
    publisher ={De Gruyter Brill} ,
    year = {2020}
}

@article{lagrange,
  title={Nouvelles recherches sur la nature et la propagation du son},
  author={Lagrange, J.L.},
  journal={Miscellanea Taurinensia},
  volume={II},
  pages={151--332},
  year={1760},
}

@Article{leray,
author="Leray, Jean",
title="Sur le mouvement d'un liquide visqueux emplissant l'espace",
journal="Acta Mathematica",
year="1934",
volume="63",
number="1",
pages="193--248",
doi="10.1007/BF02547354",
}

@article {DiPernaMajda,
    AUTHOR = {DiPerna, Ronald J. and Majda, Andrew J.},
     TITLE = {Oscillations and concentrations in weak solutions of the
              incompressible fluid equations},
   JOURNAL = {Comm. Math. Phys.},
  FJOURNAL = {Communications in Mathematical Physics},
    VOLUME = {108},
      YEAR = {1987},
    NUMBER = {4},
     PAGES = {667--689},
       DOI = {10.1007/BF01214424},
}

@book {lionsbook,
    AUTHOR = {Lions, Pierre-Louis},
     TITLE = {Mathematical topics in fluid mechanics. {V}ol. 1},
 PUBLISHER = {The Clarendon Press, New York},
      YEAR = {1996},
     PAGES = {xiv+237},
}

@article{convexintegration,
 author = {De Lellis, Camillo and Sz{\'e}kelyhidi, L{\'a}szl{\'o} jun.},
 title = {The {Euler} equations as a differential inclusion},
 fjournal = {Annals of Mathematics. Second Series},
 journal = {Ann. Math. (2)},
 volume = {170},
 number = {3},
 pages = {1417--1436},
 year = {2009},
 doi = {10.4007/annals.2009.170.1417},
}

@book {raymond,
    AUTHOR = {Ars\'{e}nio, Diogo and Saint-Raymond, Laure},
     TITLE = {From the {V}lasov--{M}axwell--{B}oltzmann system to
              incompressible viscous electro-magneto-hydrodynamics. {V}ol.
              1},
    SERIES = {EMS Monographs in Mathematics},
 PUBLISHER = {European Mathematical Society (EMS), Z\"{u}rich},
      YEAR = {2019},
     PAGES = {xii+406},
      DOI = {10.4171/193},
}

@article{VladTristan,
 author = {Buckmaster, Tristan and Vicol, Vlad},
 title = {Nonuniqueness of weak solutions to the {Navier}-{Stokes} equation},
 fjournal = {Annals of Mathematics. Second Series},
 journal = {Ann. Math. (2)},
 volume = {189},
 number = {1},
 pages = {101--144},
 year = {2019},
 doi = {10.4007/annals.2019.189.1.3},
}

@article{dallas,
 author = {Albritton, Dallas and Bru{\'e}, Elia and Colombo, Maria},
 title = {Non-uniqueness of {Leray} solutions of the forced {Navier}-{Stokes} equations},
 fjournal = {Annals of Mathematics. Second Series},
 journal = {Ann. Math. (2)},
 volume = {196},
 number = {1},
 pages = {415--455},
 year = {2022},
 doi = {10.4007/annals.2022.196.1.3},
}

@misc{selectEuler,
      title={Maximal dissipation and well-posedness of the Euler system of gas dynamics}, 
      author={Eduard Feireisl and Ansgar Jüngel and Mária Lukáčová-Medvid'ová},
      year={2025},
      eprint={2501.05134},
      archivePrefix={arXiv},
}

@misc{RobVari,
      title={Energy-variational structure in evolution equations}, 
      author={Robert Lasarzik},
      year={2025},
      eprint={2503.11438},
      archivePrefix={arXiv},
}

@article{viscoenvar,
 author = {Agosti, Abramo and Lasarzik, Robert and Rocca, Elisabetta},
 title = {Energy-variational solutions for viscoelastic fluid models},
 fjournal = {Advances in Nonlinear Analysis},
 journal = {Adv. Nonlinear Anal.},
 volume = {13},
 pages = {35},
 note = {Id/No 20240056},
 year = {2024},
 doi = {10.1515/anona-2024-0056},
}

@article{envar,
 author = {Lasarzik, Robert},
 title = {On the existence of energy-variational solutions in the context of multidimensional incompressible fluid dynamics},
 fjournal = {Mathematical Methods in the Applied Sciences},
 journal = {Math. Methods Appl. Sci.},
 volume = {47},
 number = {6},
 pages = {4319--4344},
 year = {2024},
 doi = {10.1002/mma.9816},
}

@article{Max,
 author = {Lasarzik, Robert and Reiter, Maximilian E. V.},
 title = {Analysis and numerical approximation of energy-variational solutions to the {Ericksen}-{Leslie} equations},
 fjournal = {Acta Applicandae Mathematicae},
 journal = {Acta Appl. Math.},
 volume = {184},
 pages = {44},
 note = {Id/No 11},
 year = {2023},
 doi = {10.1007/s10440-023-00563-9},
}

@article{maxdiss,
 author = {Lasarzik, Robert},
 title = {Maximally dissipative solutions for incompressible fluid dynamics},
 fjournal = {ZAMP. Zeitschrift f{\"u}r angewandte Mathematik und Physik},
 journal = {Z. Angew. Math. Phys.},
 volume = {73},
 number = {1},
 pages = {21},
 note = {Id/No 1},
 year = {2022},
 doi = {10.1007/s00033-021-01628-1},
}

@article{West,
 author = {Westdickenberg, Michael},
 title = {Minimal acceleration for the multi-dimensional isentropic {Euler} equations},
 fjournal = {Archive for Rational Mechanics and Analysis},
 journal = {Arch. Ration. Mech. Anal.},
 volume = {247},
 number = {3},
 pages = {37},
 note = {Id/No 35},
 year = {2023},
 doi = {10.1007/s00205-023-01864-x},
}

@misc{EmilMaiximalTurb,
      title={Maximal turbulence as a selection criterion for measure-valued solutions}, 
      author={Christian Klingenberg and Simon Markfelder and Emil Wiedemann},
      year={2025},
      eprint={2503.20343},
      archivePrefix={arXiv},
}

@article{leastaction,
 author = {Gimperlein, H. and Grinfeld, M. and Knops, R. J. and Slemrod, M.},
 title = {The least action admissibility principle},
 fjournal = {Archive for Rational Mechanics and Analysis},
 journal = {Arch. Ration. Mech. Anal.},
 volume = {249},
 number = {2},
 pages = {17},
 note = {Id/No 22},
 year = {2025},
 doi = {10.1007/s00205-025-02094-z},
}

@ARTICLE{MielkeIntro,
   author = {Mielke, Alexander},
    title = "{An introduction to the analysis of gradient systems}",
  journal = {{WIAS} Preprint, No. 3022, Berlin},
     year = 2023,
      doi = {10.20347/wias.preprint.3022},
}

@article{MCF,
 author = {Hensel, Sebastian and Laux, Tim},
 title = {A new varifold solution concept for mean curvature flow: convergence of the {Allen}-{Cahn} equation and weak-strong uniqueness},
 fjournal = {Journal of Differential Geometry},
 journal = {J. Differ. Geom.},
 volume = {130},
 number = {1},
 pages = {209--268},
 year = {2025},
 doi = {10.4310/jdg/1747065796},
}

@article{japMeasVal,
 author = {B{\v{r}}ezina, Jan and Feireisl, Eduard},
 title = {Measure-valued solutions to the complete {Euler} system},
 fjournal = {Journal of the Mathematical Society of Japan},
 journal = {J. Math. Soc. Japan},
 volume = {70},
 number = {4},
 pages = {1227--1245},
 year = {2018},
 doi = {10.2969/jmsj/77337733},
}

@article{binormCurveSol,
 author = {Banica, Valeria and Vega, Luis},
 title = {Riemann's non-differentiable function and the binormal curvature flow},
 fjournal = {Archive for Rational Mechanics and Analysis},
 journal = {Arch. Ration. Mech. Anal.},
 volume = {244},
 number = {2},
 pages = {501--540},
 year = {2022},
 doi = {10.1007/s00205-022-01769-1},
}

@article{gallenmüller2023cahnhillardkellersegelsystemshighfriction,
 author = {Gallenm{\"u}ller, Dennis and Gwiazda, Piotr and {\'S}wierczewska-Gwiazda, Agnieszka and Wo{\'z}nicki, Jakub},
 title = {Cahn-Hillard and {Keller}-{Segel} systems as high-friction limits of {Euler}-{Korteweg} and {Euler}-{Poisson} equations},
 fjournal = {Calculus of Variations and Partial Differential Equations},
 journal = {Calc. Var. Partial Differ. Equ.},
 volume = {63},
 number = {2},
 pages = {27},
 note = {Id/No 47},
 year = {2024},
 doi = {10.1007/s00526-023-02656-7},
}

@article{Chiodaroli_2014,
 author = {Chiodaroli, Elisabetta and De Lellis, Camillo and Kreml, Ond{\v{r}}ej},
 title = {Global ill-posedness of the isentropic system of gas dynamics},
 fjournal = {Communications on Pure and Applied Mathematics},
 journal = {Commun. Pure Appl. Math.},
 volume = {68},
 number = {7},
 pages = {1157--1190},
 year = {2015},
 doi = {10.1002/cpa.21537},
}

@article{Donatelli_2014,
 author = {Donatelli, Donatella and Feireisl, Eduard and Marcati, Pierangelo},
 title = {Well/ill posedness for the {Euler}-{Korteweg}-{Poisson} system and related problems},
 fjournal = {Communications in Partial Differential Equations},
 journal = {Commun. Partial Differ. Equations},
 volume = {40},
 number = {7},
 pages = {1314--1335},
 year = {2015},
 doi = {10.1080/03605302.2014.972517},
}

@book{chipot,
 author = {Chipot, Michel},
 title = {Elements of nonlinear analysis},
 year = {2000},
 publisher = {Basel: Birkh{\"a}user},
}

@book{evansgariepy,
 author = {Evans, Lawrence Craig and Gariepy, Ronald F.},
 title = {Measure theory and fine properties of functions},
 edition = {2nd edition},
 fseries = {Textbooks in Mathematics},
 series = {Textb. Math.},
 year = {2015},
 publisher = { CRC Press},
 doi = {10.1201/9781003583004},
}

 \end{document}